\documentclass[11pt]{amsart}

\usepackage{amssymb,amsmath,amsbsy,amsfonts,mathrsfs,amscd,eucal,bm,tcolorbox,stmaryrd}
\usepackage{subcaption}

\usepackage[colorlinks,citecolor=cyan]{hyperref}

\newtheorem{theorem}{Theorem}[section]
\newtheorem{lemma}[theorem]{Lemma}

\newtheorem{definition}[theorem]{Definition}
\newtheorem{remark}[theorem]{Remark}

\newcommand{\calP}{\mathcal{P}}

\newcommand{\cald}{d}
\newcommand{\calD}{\mathcal{D}}
\newcommand{\calI}{\mathcal{I}}
\newcommand{\bx}{\boldsymbol{x}}

\newcommand{\frakh}{\mathfrak{h}}

\newcommand{\bR}{\mathbb{R}}

\newcommand{\dive}{{\ensuremath\mathop{\mathrm{div}\,}}}
\newcommand{\hypothesis}{{\ensuremath\mathop{\mathrm{Assumption~(A1)}}}}

\title[Bound-preserving nodal values]{A nodally bound-preserving finite element method}

\author[G.R. Barrenechea]{Gabriel R. Barrenechea}
\address[G.R.B.]{Department of Mathematics and Statistics, University of Strathclyde, 26 Richmond Street, Glasgow G1 1XH, Scotland. {\tt gabriel.barrenechea@strath.ac.uk} } 

\author[E. Georgoulis]{Emmanuil Georgoulis}
\address[E.G.]{1.
  School of Mathematical \& Computer Sciences,
	Heriot Watt University, United Kingdom, 
	2. Department of Mathematics, School of Applied Mathematical and Physical Sciences, National Technical University of Athens, Greece, and
	3. IACM-FORTH, Greece {\tt E.Georgoulis@hw.ac.uk}} 

\author[T. Pryer]{Tristan Pryer}
\address[T.P.]{Department of Mathematical Sciences, University of Bath, Claverton down, Bath BA2 7AY, UK {\tt tmp38@bath.ac.uk} } 

\author[A. Veeser]{Andreas Veeser}
\address[A.V.]{Dipartimento di Matematica, Universit\`a degli Studi di Milano, Via C. Saldini 50, 20133 Milano, Italy
 {\tt andreas.veeser@unimi.it} }

\begin{document}

\maketitle

\begin{abstract}
This work proposes a nonlinear finite element method whose nodal values preserve bounds known for the exact solution. The discrete problem involves a nonlinear projection operator mapping arbitrary nodal values into bound-preserving ones and seeks the numerical solution in the range of this projection. As the projection is not injective, a stabilisation based upon the complementary projection is added in order to restore well-posedness. Within the framework of elliptic problems, the discrete problem may be viewed as a reformulation of a discrete obstacle problem, incorporating the inequality constraints through Lipschitz projections.
	
The derivation of the proposed method is exemplified for linear and nonlinear reaction-diffusion problems.  {Near-best approximation results in suitable norms are established. In particular, we prove that, in the linear case, the numerical solution is the best approximation in the energy norm among all nodally bound-preserving finite element functions.}

A series of numerical experiments for such problems showcase the good behaviour of the proposed bound-preserving finite element method. % over the classical one.
\end{abstract}

%%%%%%%%%%%%%%%%%%%%%%%%%%%%%%%%%%%%%
%%%%%%%%%%%%%%%%%%%%%%%%%%%%%%%%%%%%%
\section{Introduction}

Structure-preserving numerical methods have been an overarching theme in computational partial differential equations (PDEs) over the
years. By structure-pre\-serv\-ing we mean methods that produce approximations satisfying certain desired properties of the underlying exact problem, e.g., local conservation,  entropy
inequalities, maximum principle, pointwise divergence-free constraints, or exactly
symmetric stress tensor approximations, just to name a few.

Numerical methods satisfying {\it Discrete Maximum Principles (DMP)} and/or {\it monotonicity} properties have been studied extensively in the finite element literature; see \cite{Ciarlet70,CR73,Kik77,MH85,BE05,BKK08, BBK17,BJK17} for a (very) non-exhaustive list, and \cite{BJK22} for a recent review and more related references. Methods satisfying the latter properties imply {\it positivity preservation} or, more generally, {\it bound preservation}  of the resulting numerical solutions. 

Crucially, however, bound-preservation appears to be a weaker structural requirement than the DMP. % and monotonicity.  
Nonetheless, bound-preserving numerical solutions are indispensable for the numerical stability of many complex nonlinear phenomena modelled by systems of PDEs. For instance, many such PDE models are valid only for positive solutions of the  constituent equations. Examples include nonlinear reaction-diffusion systems modelling concentrations of reactants, or turbulence-inducing fields. Also, phase-field PDE models are usually characterised by solutions satisfying pointwise global maxima and minima. Although failure to preserve the same bounds for the numerical solutions is not typically catastrophic for scalar PDE problems, they may have a compounding effect when such PDEs are part of a more complex system of equations.

There is thus an interest in bound-preserving methods which, when appropriately designed, may be less ``stiff''    than current schemes striving to satisfy the DMP or monotonicity.

This work proposes a new finite element method in this spirit. To this end, assume that we are given a partial differential equation (with its associated weak form), and  a Lagrange finite element space of arbitrary order.
Moreover, for simplicity, assume that the exact solution is positive, i.e.\ the zero function is a lower bound.
Then the formal derivation of the new method can be summarized by the following three steps:
\begin{itemize}
	\item Introduce a projection operator %%P$
	$v_h \mapsto v_h^+$ on the finite element space so that the nodal values of the projected function $v_h^+$ are positive.
	\item As usual, replace the test space in the weak formulation of the continuous problem with the finite element space, but, instead of looking for a finite element function as the numerical solution, seek for a projected one in the form %$Pu_h$
	$u_h^+$.
	\item Add a stabilisation based upon the complementary projection %$I-P$
	$v_h \mapsto v_h^-$ onto the finite element space.
\end{itemize}
Notice that the projection operator and so the proposed method are nonlinear. The latter is expected in light of the classical ``Godunov barrier'' principle. The solution of the proposed method is sought in the range of the projection, where the projection acts as a parametrisation and the parameter domain is the finite element space. The fact that the projection is not injective entails the need of the stabilisation in the third step.

Interestingly, if its nodal values are  positive, the classical finite element solution coincides with the one of the proposed method. More precisely, and more importantly, the solution to the present method
%latter solution 
turns out to be the orthogonal projection of the exact solution 
onto the closed and convex set of finite element functions with positive
nodal values. This observation is instrumental in this work and yields the following two important consequences:  
\begin{itemize}
	\item The proposed method can be viewed as a reformulation of a discrete obstacle problem. In contrast to the latter, the reformulation consists of a discrete variational equality and encodes the crucial inequality constraints in Lipschitz-continuous projections.
	This difference may prepare the ground for solvers that are an alternative to the usually used constrained optimisation techniques;  see, e.g., \cite{EHS09}. Here, we employ a 
simple approach based upon Richardson-like iterations.
	\item Near best approximation for suitable error notions follows as a corollary. In particular, for the linear reaction-diffusion problem, the solution of the proposed method is actually the best approximation to the exact solution in the energy norm from the aforementioned convex subset. For a more general, semilinear problem with a power-like nonlinear reaction, the proposed method is shown to be quasi-optimal with respect to the sum of the $H^1$-seminorm and a quasinorm associated with the reaction term.
\end{itemize}
To illustrate both aspects, a series of numerical experiments is presented, showcasing the superior performance of the proposed method compared to standard finite element approximations.

It is worth mentioning that related approaches (albeit with distinct differences to the present one) have been advocated in the literature.
In \cite{Bochev20} the bounds on the continuous solution are imposed as a restriction in a constraint optimisation problem, and a link to
nonlinear stabilisation is presented.  Also, in \cite{MN16} a constrained optimisation problem involving a mixed  weak formulation is employed to enforce  bound-preservation.  In \cite{BE17} the link between positivity-preservation and the contact problem
is used to motivate a nonlinear stabilised method  that enforces the positivity of the solution in a weak way.
In addition, the
\textit{cut-off} finite element method \cite{Lu2013TheCM}
truncates the finite element function \textit{after}  it  is computed at a given time step, so as
to input the truncated function as approximation of the current time step;
see also \cite{YYZ22} for an application of a related idea to the Allen-Cahn equation.
 In the steady state case, the idea of \textit{truncating}
the finite element solution to respect given bounds has been justified for reaction-diffusion equations in \cite{Kreuzer} using energy arguments. 
In \cite{LS08} a conservative recovery strategy is proposed and tested numerically.
Finally, in the context of the Joule heating problem a truncation of one of the variables is introduced in order to
regularise a rough right-hand side in \cite{JM13}.   

The remainder of the manuscript is organised as follows. The rest of introduction is devoted to setting up notations and main
assumptions of this work. The finite element method is presented in Section~\ref{Sec:Fem}, and its stability and optimal error
estimates are proven. Then, in Section~\ref{Sec:nonlinear-reac} we extend this framework to a nonlinear reaction-diffusion equation, and
finally Section~\ref{Sec:Numer} is devoted to presenting numerical experiments.

\section{General setting and the linear model problem}

We will use standard notation for Sobolev spaces, in line with, e.g., \cite{EG21-I}. More precisely,
for $D\subseteq\bR^d$,  {$d=1,2,3$}, we denote by $\|\cdot\|_{0,p,D}^{}$ the $L^p(D)$-norm, when $p=2$
the subscript $p$ will be often omitted, and we only write $\|\cdot\|_{0,D}^{}$. In addition, for $s\ge 0, p\in (1,\infty)$,
 we denote by $\|\cdot\|_{s,p,D}^{}$ ($|\cdot|_{s,p,D}^{}$) the norm (seminorm) in $W^{s,p}(D)$; when
$p=2$, we will again often omit the subscript $p$. 
In addition,
we denote by
$H^{-1}(D)$ the dual of $H^1_0(D)$
{
	while identifying $L^2(D)$ with its dual. Thus, writing $\langle\cdot,\cdot\rangle_D^{}$ for the duality pairing, we have
	\begin{equation*}
		\langle f,v\rangle_\Omega^{}=\int_\Omega f(\boldsymbol{x})v(\boldsymbol{x})\textrm{d}\boldsymbol{x}\qquad
		\forall\, v\in H^1_0(\Omega)\,,
	\end{equation*}
	whenever $f\in H^{-1}(D)$ is regular enough.
}The  $L^2(D)$-inner product is denoted by $(\cdot,\cdot)_D^{}$. 
We will not distinguish between inner product and duality pairing for scalar and for vector-valued functions.

The boundary-value problems we will be concerned with are posed in an open, bounded domain $\Omega$ with polyhedral Lipschitz boundary $\partial\Omega$.
Let $\calP$ be a conforming, shape-regular partition of $\Omega$ into simplices (or quadrilaterals/hexahedra) $K$.
Over $\calP$, and for $k\ge 1$, we define the finite element space
\begin{equation}
V_{\calP}^{}:=\{ v_h^{}\in C^0(\Omega):v_h^{}|_K^{}\in\mathcal{R}(K)\;\forall\,K\in\calP\}\cap H^1_0(\Omega)\,,
\end{equation}
where 
\begin{equation}
\mathcal{R}(K) = \left\{ \begin{array}{cl} \mathbb{P}_k^{}(K) & \;\textrm{if}\; K\;\textrm{is a simplex}\,, \\
\mathbb{Q}_k^{}(K) & \;\textrm{if}\; K\;\textrm{is {an affine} quadrilateral/hexahedron}\, , \end{array}\right.
\end{equation}
with $\mathbb{P}_k^{}(K)$ denoting the polynomials of total degree $k$ on $K$ and $\mathbb{Q}_k^{}(K) $  denoting the, 
mapped from a reference element through an affine mapping, polynomials of degree $\leq k$ in each variable.  
We denote by $\bx_i^{},i=1,\ldots,N$ the set of internal nodes of $\calP$, and by $\phi_1^{},\ldots,\phi_N^{}$ the set 
of usual Lagrangian basis functions spanning the space $V_\calP^{}$.

{The  diameter of $K$ is denoted by $h_K^{}$, $h=\max\{ h_K^{}:K\in\calP\}$,  and we define the mesh function $\frakh$ as a continuous, element-wise
linear function defined as a local average.  More precisely, for a node  $\bx_i^{}$ of $\calP$ we define the local neighbourhood $K_{\bx_i}^{}=\{ K\in \calP:\bx_i^{}\in K\}$, and define
$\frakh$ as the only element of $V_\calP^{}$ (with $k=1$) given by the nodal values
\begin{equation}\label{mesh-function-definition}
\frakh(\bx_i^{})=\frac{\sum_{K\in K_{\bx_i}^{}}h_K^{}}{\#K_{\bx_i}^{}}\,.
\end{equation} }

We recall the inverse inequality (see, e.g., \cite[Lemma~12.1]{EG21-I}): for all $m,\ell\in\mathbb{N}, m\le \ell$ and all $p,q\in[1,+\infty]$, there
exists a constant $C$, independent of $K \in \calP$ %\textcolor{magenta}{the simplex $K \in \calP$},
such that
\begin{equation}\label{inverse}
|q|_{\ell,p,K}^{}\le Ch_K^{m-\ell+d\left(\frac{1}{p}-\frac{1}{q}\right)}\,|q|_{m,q,K}^{}\,,
\end{equation} 
for every polynomial function $q$ defined on $K$. %A global version of this inequality can also be derived using the quasi-uniformity of the mesh family.

In the space $V_\calP^{}$ we denote by $(\cdot,\cdot)_h^{}$  the lumped $L^2(\Omega)$-inner product given by
\begin{equation}\label{lumped-product}
(v_h^{},w_h^{})_h^{}:= \sum_{i=1}^N \frakh^d(\bx_i) v_h^{}(\bx_i^{})w_h^{}(\bx_i^{})\,,
\end{equation}
with associated norm $|v_h^{}|_h^{}=(v_h^{},v_h^{})_h^{\frac{1}{2}}$.

\begin{remark} 
The following result, whose proof can be found in \cite[Proposition~12.5]{EG21-I},  
will be of importance for us in the analysis of the method
proposed in Section~\ref{Sec:Fem}:  there exist $C,c>0$, independent of $\frakh$, such that
\begin{equation}\label{equivalence}
c\,|v_h^{}|_h^2\le \|v_h^{}\|^2_{0,\Omega}\le C\,|v_h^{}|_h^2\,, 
\end{equation}
for all $v_h^{}\in V_\calP^{}$. 
\end{remark}

%%%%%%%%%%%%%%%%%%%%%%%%%%%%%%%%%%%%%
\subsection{Linear model problem}
Given $f\in H^{-1}(\Omega)$ we consider the following reaction-diffusion
equation: find $u:\Omega\to\bR$ such that
\begin{align}\label{reaction-diffusion}
-\dive\big(\calD\nabla u) +\mu u  &=f\qquad\textrm{in}\;\Omega\,,\\
u &=0 \qquad\textrm{on}\;\partial\Omega\, ,\nonumber
\end{align}
with $\mathcal{D}=\big(d_{ij}^{}\big)_{i,j=1}^d\in L^\infty(\Omega)^{d\times d}$ and $\mu\in L^\infty(\Omega)$ stand for the diffusion tensor and reaction coefficient, respectively. We assume that $\mu(\boldsymbol{x})\ge \mu_0^{}\ge 0$ a.e.\ in $\Omega$, and that the diffusion tensor $\calD$ is symmetric and uniformly strictly positive-definite in $\Omega$, viz.,
there exists a positive constant $\cald_0^{}>0$ such that, for almost all $\bx\in\Omega$, we have
\begin{equation}\label{elliptic-D}
\sum_{i,j=1}^d y_i^{} \cald_{ij}^{}(\bx)y_j^{}\ge d_0^{}\sum_{i=1}^d y_i^2\qquad\forall\, (y_1^{},\ldots,y_d^{})\in\bR^d\,.
\end{equation}

The weak formulation of \eqref{reaction-diffusion} reads: find $u\in H^1_0(\Omega)$, such that
\begin{equation}\label{weak-form}
a(u,v)=\langle f,v\rangle_\Omega^{}\qquad\forall\, v\in H^1_0(\Omega)\,,
\end{equation}
where $a(\cdot,\cdot):H^1_0(\Omega)\times H^1_0(\Omega)\to \mathbb{R}$ is the bilinear form defined by
\begin{equation}
a(v,w) :=
(\calD\nabla v, \nabla w)_\Omega^{}+ (\mu v , w)_\Omega^{}\,. \label{def-a}
\end{equation}

Defining the {\sl energy norm}
\begin{equation}\label{energy-norm}
\|v\|_a^{}:= \sqrt{a(v,v)}\,,
\end{equation}
the well-posedness of \eqref{weak-form} follows from the Lax-Milgram Lemma (see, e.g., \cite{EG21-I}). 

\begin{remark} At the core of this work is the following property of the solution of \eqref{weak-form}: 
as a consequence of maximum and comparison principles (see \cite{GT01}, or \cite[Corollary~4.4]{RR04})
the following bounds can be proven: for almost all $\bx\in \Omega$
the solution $u$ of \eqref{weak-form} satisfies 
\begin{equation}
-\frac{\|f\|_{0,\infty,\Omega}^{}}{\mu_0^{}}\le u(\bx) \le \frac{\|f\|_{0,\infty,\Omega}^{}}{\mu_0^{}}\,.
\end{equation}
 This last statement can be made more precise if $f\ge 0$ in $\Omega$. In fact, in this case for almost all $\bx\in \Omega$
the following inequality holds:
\begin{equation}
0\le u(\bx) \le \frac{\|f\|_{0,\infty,\Omega}^{}}{\mu_0^{}}\,.
\end{equation}
\end{remark}

The results given in the above remark motivate the introduction of the following assumption.

\noindent\underline{Assumption~(A1):}
We will suppose that the solution of \eqref{reaction-diffusion} satisfies
\begin{equation}\label{a-priori-bounds-u}
0\le u(x) \le \mathfrak{a}\qquad\textrm{for almost all}\; x\in\,\Omega\,,
\end{equation}
where $\mathfrak{a}$ is a known constant. 

\begin{remark}
The lower bound in $\hypothesis$ is not required to be equal to zero and 
all the results proven below hold for a more general bounding
box without major modifications. In particular,  the value $\mathfrak{a}$ can be replaced by a 
non-negative{, continuous} function $\mathfrak{a}(\bx)$.
\end{remark}

The standard Galerkin finite element method for \eqref{reaction-diffusion} is given by: find $u_h^{\rm FEM}\in V_\calP^{}$ such that
\begin{equation}\label{Gal:FEM}
a(u_h^{\rm FEM},v_h^{})=\langle f,v_h^{}\rangle_\Omega^{}\qquad\forall\, v_h^{}\in V_\calP^{}\,.
\end{equation}

It is well-known that {$u_h^{\rm FEM}$ is a best approximation in the following sense:
\begin{equation}\label{Galerkin-is-best-approx}
		\| u - u_h^{\rm FEM} \|_a 
		=
		\inf_{v_h \in V_\calP^{}} \| u - v_h \|_a. 
\end{equation}
However, this nice property does not prevent that}, when the reaction $\mu$ dominates the diffusion $\calD$, that is, e.g. if $\|\calD\|_{0,\infty,\Omega}^{} \ll \mu$, then $u_h^{\rm FEM}$ may exhibit spurious oscillations (see, e.g., \cite{RST08}) and,
thus, $u_h^{\rm FEM}$ may fail to satisfy $\hypothesis$. In fact, very stringent conditions need to be imposed on the mesh for the solution of a reaction-diffusion equation
such as \eqref{weak-form} to satisfy $\hypothesis$ (see, e.g., \cite{BKK08} for details), even in the case of scalar (isotropic) diffusion. For the case of a general diffusion $\calD$,
conditions on the mesh related to the weighted inner product $(\calD\cdot,\cdot)$ in $\bR^d$ need to be imposed, in addition to a mesh size restriction
(see, e.g., \cite{Hua11} for details).

%%%%%%%%%%%%%%%%%%%%%%%%%%%%%%%%%%%%%
%%%%%%%%%%%%%%%%%%%%%%%%%%%%%%%%%%%%%
\section{The finite element method for the linear problem}\label{Sec:Fem}

{The goal of this section is to derive a method that, on the one hand, essentially preserves the bounds in Assumption (A1) without restrictions on the mesh and, on the other hand, maintains the good approximation properties \eqref{Galerkin-is-best-approx} of the classical finite element solution.}

{To this end,}
we define the following closed convex subset of $V_\calP^{}$:
\begin{equation}\label{convex}
V_{\calP}^{+}:=\{ v_h^{}\in V_{\calP}^{}:v_h^{}(\bx_i^{})\in [0,\mathfrak{a}]\; \textrm{for all}\,\bx_i^{}=1,\ldots,N\}\,,
\end{equation}
that is, the set of finite element functions that respect the bound \eqref{a-priori-bounds-u}  at their degrees of freedom. 
With this convex set in mind, every finite element function is split as $v_h^{}=v_h^++v_h^-$, where $v_h^+\in V_\calP^+$ is defined by
\begin{equation} \label{positive-part}
v_h^+:= \sum_{i=1}^N \max\big\{ 0,\min\{v_h^{}(\bx_i^{}),\mathfrak{a}\}\big\}\phi_i^{}\,,
\end{equation}
and
\begin{equation}
v_h^-:=  v_h^{}-v_h^+\,, \label{negative-part}
\end{equation}
is the part of the function residing outside $V_\calP^+$.
{From now on, we refer to the functions $v_h^+$ and $v_h^-$ as the \textit{constrained}
and \textit{complementary} parts of $v_h^{}$.}
%With a slight abuse of terminology, the functions $v_h^+$ and $v_h^-$ will be referred to as the positive and negative parts of $v_h^{}$, respectively, although, strictly speaking, $v_h^+(\bx)\in[0,\mathfrak{a}]$.

The finite element method proposed in this work reads: find $u_h^{}\in V_\calP^{}$, such that
\begin{equation}\label{FEM}
a_h^{}(u_h^{};v_h^{}) =  \langle f, v_h^{}\rangle_\Omega^{}\qquad\forall\, v_h^{}\in V_\calP^{}\,,
\end{equation}
where $a_h^{}(\cdot;\cdot)$ is the nonlinear form given by
\begin{equation}
a_h^{}(u_h^{};v_h^{}) := a(u_h^{+},v_h^{}) +s(u_h^{-},v_h^{}) \,,
\end{equation}
with $a(\cdot,\cdot)$ defined in \eqref{def-a}, and $s(\cdot,\cdot):C(\bar{\Omega})\times C(\bar{\Omega})\to \mathbb{R}$ is the stabilising bilinear form defined by
\begin{equation}\label{def-s}%\label{stab:non-quasi}
	s(v_h^{},w_h^{}) =\alpha\sum_{i=1}^N \big(\|\calD\|_{0,\infty,\omega_i}^{}\frakh(\bx_i^{})^{d-2}+
	\|\mu\|_{0,\infty,\omega_i}^{}\frakh(\bx_i^{})^{d}\big)\,v_h^{}(\bx_i^{})w_h^{}(\bx_i^{})\,,
\end{equation}
where 
$\alpha>0$ is a non-dimensional constant to be determined precisely {in Theorem~\ref{Theo:Well-posedness}}, and %$\omega_i$ refers to the vertex neighbourhood of the node $\bx_i$, i.e., 
$\omega_i:=\cup_{K\in \calP: K\cap K_{\bx_i}^{}\neq \emptyset}K$ { denotes an extended patch. The definition of $\omega_i$ will be exploited in establishing \eqref{s.ge.a} below.}

Defining the  stabilisation norm as
\begin{equation}\label{norms}
\|v_h^{}\|_s^{}:=\sqrt{ s(v_h^{},v_h^{})}\,,
\end{equation}
and using \eqref{equivalence}, { there}
exists a $C_{\rm equiv}^{}>0$, depending only on the shape-regularity constant, such that
\begin{equation}\label{s.ge.a}
 \|v_h^{}\|_a^2\le \frac{C_{\rm equiv}^{}}{\alpha}\, \|v_h^{}\|_s^2\qquad\forall\, v_h^{}\in V_\calP^{}\,,
\end{equation}
where $\alpha>0$ is the stabilisation parameter appearing in the definition \eqref{def-s} of $s(\cdot,\cdot)$.

~ \\

\subsection{Well-posedness and consistency}
~ \\

In this section we will analyse the existence, uniqueness, and stability results for the proposed method~\eqref{FEM}. We start with
the following monotonicity result for the stabilising form $s(\cdot,\cdot)$.

\begin{lemma}\label{Lem:s-monotone}
The bilinear form $s(\cdot,\cdot)$ satisfies the following inequalities:
\begin{align}
s(v_h^{-}-w_h^{-}, v_h^{+}- w_h^{+}) &\ge 0\qquad \forall\, v_h^{},w_h^{}\in V_\calP^{}\,, \label{s-monotone-2}\\
s(v_h^-,w_h^{}-v_h^+) &\le 0\qquad\forall\, v_h^{}\in V_\calP^{}, w_h^{}\in V_\calP^+\,. \label{s-monotone-1}
\end{align}
\end{lemma}

\begin{proof} 
Let $\bx_i^{}\in\calP$ be any internal node.  If 
$v_h^{}(\bx_i^{})\ge w_h^{}(\bx_i^{})$, then $v_h^{+}(\bx_i^{})\ge w_h^{+}(\bx_i^{})$. Moreover, if $v_h^{}(\bx_i^{})\ge w_h^{}(\bx_i^{})>\mathfrak{a}$, then 
$v_h^{-}(\bx_i^{})=v_h^{}(\bx_i^{})-\mathfrak{a}\ge w_h^{}(\bx_i^{})-\mathfrak{a} = w_h^{-}(\bx_i^{})$. If $v_h^{}(\bx_i^{})>\mathfrak{a}$ and 
$w_h^{}(\bx_i^{})\le \mathfrak{a}$, then
$v_h^{-}(\bx_i^{})=v_h^{}(\bx_i^{})-\mathfrak{a}\ge 0 \ge w_h^{-}(\bx_i^{})$. If $v_h^{}(\bx_i^{})\le \mathfrak{a}$, then either $v_h^{-}(\bx_i^{})=0\ge w_h^{-}(\bx_i^{})$,
or $v_h^{-}(\bx_i^{})= v_h^{}(\bx_i^{}) \ge w_h^{}(\bx_i^{}) = w_h^{-}(\bx_i^{})$. 
Thus
	\begin{align}
&s(v_h^{-}-w_h^{-}, v_h^{+}- w_h^{+}) \nonumber\\
=&\	\alpha\sum_{i=1}^N \big(\|\calD\|_{0,\infty,\omega_i}^{}\frakh(\bx_i^{})^{d-2}+
	\|\mu\|_{0,\infty,\omega_i}^{}\frakh(\bx_i^{})^{d}\big)\,(v_h^{-}-w_h^{-})(\bx_i^{})( v_h^{+}- w_h^{+})(\bx_i^{}) \nonumber\\
	\ge&\  0,
\end{align}
which proves \eqref{s-monotone-2}. Taking $w_h^{}\in V_\calP^+$ gives $w_h^-=0$, and then \eqref{s-monotone-1} follows from \eqref{s-monotone-2}.
\end{proof}

\begin{theorem}[Well-posedness]\label{Theo:Well-posedness}
Let $T:V_\calP^{}\to [V_\calP^{}]'$ be the mapping defined by
\begin{equation}\label{T-definition}
[Tv_h^{},w_h^{}]=a(v_h^+,w_h^{})+s(v_h^-,w_h^{})\,.
\end{equation}
Then, $T$ is  continuous and, if the non-dimensional parameter $\alpha$ is chosen such that $\alpha\ge C_{\rm equiv}^{}$, it is also {strongly monotone,} since then  $T$ satisfies: {there exists $\beta>0$, independent
of $h$, such that
\begin{equation}\label{T:monotone}
[Tv_h^{}-Tw_h^{},v_h^{}-w_h^{}]\ge
\beta\, \|v_h^{}-w_h^{}\|^2_a\,,
%\frac{1}{2}
%\big(\|v_h^{+}-w_h^{+}\|^2_a  +  \|v_h^{-}-w_h^{-}\|_s^2 \big)\,,
\end{equation}}
for all $v_h^{},w_h^{}\in V_\calP^{}$.
As a consequence, \eqref{FEM} has a unique solution $u_h^{}\in V_\calP^{}$.
\end{theorem}

\begin{proof}  
 We start defining the mesh-dependent norm $\|\cdot\|_\calP^{}$ by
\begin{equation}\label{mesh-norm}
%\|v_h^{}\|_\calP^{}:= \left\{ \|v_h^{+}\|_a^{2}+\|v_h^{-}\|_s^{2}\right\}^\frac{1}{2}\,,
\|v_h^{}\|_\calP^{}:= \left\{ \|v_h^{}\|_a^{2}+\|v_h^{}\|_s^{2}\right\}^\frac{1}{2}\,.
\end{equation}
Then, for all $v_h^{},w_h^{},z_h^{}\in V_\calP^{}$ using the Cauchy-Schwarz inequality and \eqref{s.ge.a}  we get to
\begin{align}
[Tv_h^{}-Tw_h^{},z_h^{}] &= a\big(v_h^+-w_h^+,z_h^{}\big)+ s\big(v_h^- -w_h^{-},z_h^{}\big) \nonumber\\
&\le \big (\|v_h^+-w_h^+\|_a^{2}+\|v_h^- -w_h^{-}\|_{s}^{2}\big)^{\frac{1}{2}}\,\|z_h^{}\|_\calP^{}\,,%\,\big(\|z_h^{}\|_a^{2}+\|z_h^{}\|_s^{2}\big)^\frac{1}{2}\,.%\nonumber\\
%&\le C\, \big (\|v_h^+-w_h^+\|_a^{2}+\|v_h^- -w_h^{-}\|_{s}^{2}\big)^{\frac{1}{2}}\,\|z_h^{}\|_\calP^{}\,,
\end{align}
which proves the continuity of $T$. %(the fact that the constant $C>0$ in the above inequality depends on negative powers of $h$ does not  affect the result).
To prove the monotonicity of $T$, let $v_h^{},w_h^{}\in V_\calP^{}$. Using the Cauchy-Schwarz and Young inequalities, and 
\eqref{s.ge.a}, we obtain
{
\begin{align}
&[Tv_h^{}-Tw_h^{}, v_h^{}- w_h^{}] = a(v_h^{+}-w_h^{+}, v_h^{}- w_h^{}) + s(v_h^{-}-w_h^{-}, v_h^{}- w_h^{}) \nonumber\\
&= \|v_h^{+}-w_h^{+}\|_a^2  + a(v_h^{+}-w_h^{+}, v_h^{-}- w_h^{-})  + \|v_h^{-}-w_h^{-}\|_s^2  + s(v_h^{-}-w_h^{-}, v_h^{+}- w_h^{+}) \nonumber\\
&\ge  \frac{1}{2}\|v_h^{+}-w_h^{+}\|^2_a  + \left( 1-\frac{C_{\rm equiv}^{}}{2\alpha}\right) \|v_h^{-}-w_h^{-}\|_s^2 + s(v_h^{-}-w_h^{-}, v_h^{+}- w_h^{+})\nonumber\\
&\ge  \frac{1}{2}\|v_h^{+}-w_h^{+}\|^2_a  +\left( 1-\frac{C_{\rm equiv}^{}}{2\alpha}\right) \|v_h^{-}-w_h^{-}\|_s^2\,,  \label{extra-0}
\end{align}
where in the last inequality we used \eqref{s-monotone-2} in Lemma~\ref{Lem:s-monotone}.
Next,  using that $\alpha\ge C_{\rm equiv}^{}$ and \eqref{s.ge.a} we arrive at
\begin{equation}\label{extra-1}
\left( 1-\frac{C_{\rm equiv}^{}}{2\alpha}\right) \|v_h^{-}-w_h^{-}\|_s^2\ge
\frac{\alpha}{2C_{\rm equiv}^{}} \|v_h^{-}-w_h^{-}\|_a^2\,.
\end{equation}
Hence,  \eqref{T:monotone} follows replacing \eqref{extra-1} in \eqref{extra-0} and the fact that, since
$a(\cdot,\cdot)$ is symmetric and elliptic, we have $\|v_h^{}\|_a^{}\le \sqrt{2}\,(\|v_h^{+}\|_a^{}+\|v_h^{-}\|_a^{})$.
}

{Finally, the existence and uniqueness of solutions follows by using classical results in monotone operator theory (see, e.g., \cite[Theorem~10.49]{RR04}).}
%
%and \eqref{T:monotone} follows by applying \eqref{s-monotone-2} in Lemma~\ref{Lem:s-monotone} and the fact that $\alpha\ge C_{\rm equiv}^{}$. Finally, the existence and uniqueness of solutions follows by noticing that since $T0=0$, then $T$ is also coercive.
%So, standard fixed point results (see, e.g., \cite[Theorem~10.49]{RR04}) give the existence of solutions, while the uniqueness follows from \eqref{T:monotone}. 
\end{proof}

For the next observation, it is useful to recall that, for Galerkin methods, consistency can be expressed as an invariance property of the operator mapping the exact solution to the Galerkin one.

\begin{lemma}[Consistency]
\label{L:consistency}
Under $\hypothesis$, the method \eqref{FEM} enjoys the following invariance property: if the exact solution $u$ belongs to $V_\calP{}$, then $u_h^+=u_h^{}=u$.
\end{lemma}

\begin{proof}
As $u \in V_\calP{}$, the functions	$u^+$ and $u^-$ are defined and, due to $\hypothesis$, we have $u^+=u$
and $u^-=0$. As a result, we get 
\begin{equation*}%\label{consistency}
 a_h^{}(u;v_h^{}) = a(u^{+},v_h^{})=a(u,v_h^{}) =\langle f,v_h\rangle_\Omega^{}\,,
\end{equation*}
for all $v_h\in V_\calP^{}$. Thanks to Theorem \ref{Theo:Well-posedness}, the solution of \eqref{FEM} is unique and therefore we obtain $u_h=u$. Using that $u_h^-=u^-=0$, we conclude that  $u_h^+ = u_h^{}$.
\end{proof}

\begin{remark}
  { Whenever $k=1$ the method~\eqref{FEM} is actually
    bound-preserving throughout the domain $\Omega$, not only in the
    nodes. Furthermore, if $\mu>0$, in addition it also respects the
    discrete maximum principle. That is, if $f\ge 0$ then $u_h^+$
    cannot attain an interior negative minimum, and reaches its
    minimum at the boundary.  If $\mu=0$, it appears that the method
    by construction does not guarantee that $u_h^+$ cannot attain an
    interior minimum. Nonetheless, all our numerical experiments to
    date have failed to produce such a case.  }
\end{remark}

%%%%%%%%%%%%%%%%%%%%%%%%%%%%%%%%%%%%%

\subsection{Characterisation of the {constrained} part and error estimates}\label{Sec:error}
~ \\

One of the salient features of the method~\eqref{FEM} is that $u_h^+$
is characterised as being the unique solution of variational
inequality posed on the closed convex set $V_\calP^+$. This is proven
in the next result.

\begin{theorem}[{Characterization of constrained part}]\label{Th:Variational-Inequality}
Let $u_h^{}\in V_\calP^{}$ be the unique solution of \eqref{FEM}. Then, $u_h^+$ satisfies the following variational inequality: $u_h^+\in V_\calP^+$ and satisfies
\begin{equation}\label{Eq:Variational-Inequality}
a(u_h^+,v_h^{}-u_h^+)\ge \langle f, v_h^{}-u_h^+\rangle_\Omega^{}\qquad\forall\, v_h^{}\in V_\calP^+\,.
\end{equation}
\end{theorem}

\begin{proof} Using \eqref{FEM} with $v_h^{}\in V_\calP^{}$, $v_h^{-}$, and $u_h^+$ as test functions we have the following
equalities
\begin{align*}
a(u_h^+,v_h^{})+ s(u_h^-,v_h^{}) & = \langle f, v_h^{}\rangle_\Omega^{} \,,\\
a(u_h^+,v_h^{-})+ s(u_h^-,v_h^{-}) & = \langle f, v_h^{-}\rangle_\Omega^{} \,,\\
a(u_h^+,u_h^{+})+ s(u_h^-,u_h^{+}) & = \langle f, u_h^{+}\rangle_\Omega^{} \,.
\end{align*}
Subtracting the second and third equation from the first one, and using that $v_h^+=v_h^{}-v_h^-$ we get 
\begin{equation}
a(u_h^+,v_h^{+}-u_h^+) + s(u_h^-,v_h^{+}-u_h^+)  = \langle f, v_h^{+}-u_h^+\rangle_\Omega^{}\,,
\end{equation}
for all $v_h^{+}\in V_\calP^+$. Finally, using \eqref{s-monotone-1} in Lemma~\ref{Lem:s-monotone} we get that $u_h^+\in V_\calP^+$ 
satisfies \eqref{Eq:Variational-Inequality}, thus completing the proof. 
\end{proof}

\begin{remark}
Due to Stampacchia's Theorem, problem \eqref{Eq:Variational-Inequality} can be proven directly to have a unique solution.
This proves that both the stabilised method \eqref{FEM} and the variational inequality \eqref{Eq:Variational-Inequality} are, in fact, equivalent.
\end{remark}

The equivalence of the method as a  variational inequality enables us to prove best approximation error estimates in a standard fashion.

\begin{theorem}[{Abstract error analysis}]\label{Th:error}
Let $u$ be the solution of \eqref{reaction-diffusion} and $u_h\in V_\calP^{}$  be the unique solution of \eqref{FEM}. Then,
\begin{equation}\label{error-u_h^+}
\|u-u_h^+\|_a^{} = \min_{v_h^{}\in V_\calP^+}\|u-v_h^{}\|_a^{}\,.
\end{equation}
Moreover, let $u_h^{\rm FEM}$ be the solution of the (standard) finite element method \eqref{Gal:FEM}. Then,
the negative part $u_h^-$ satisfies the following error estimate
% there exists $C>0$ such that
\begin{equation}\label{error-u_h^-}
s(u_h^-,u_h^-)^{\frac{1}{2}}\le \sqrt{\frac{C_{\rm equiv}^{}}{\alpha}}\min\big\{\|u-u_h^+\|_a^{},\|u_h^{\rm FEM}-u_h^+\|_a^{}\big\}\,.
\end{equation}
\end{theorem}

\begin{proof}
Since $u_h^+$ is the solution of \eqref{Eq:Variational-Inequality}, it satisfies in particular
\begin{equation}
a(u_h^{+}-u,v_h^{}-u_h^+)
\ge 
0\qquad\forall\, v_h^{}\in V_\calP^+\,.
\end{equation}
Thus, $u_h^+$ is the best approximation of $u$ in $V_\calP^{+}$ with respect to the norm induced by
$a(\cdot,\cdot)$, thus proving the quasi-optimality for $u_h^+$ stated in \eqref{error-u_h^+}.

The proof of \eqref{error-u_h^-} is similar. In fact,  using the Cauchy-Schwarz  inequality and \eqref{s.ge.a} we get
\begin{equation}
s(u_h^-,u_h^-)=a(u-u_h^+,u_h^-)\le
\sqrt{\frac{C_{\rm equiv}^{}}{\alpha}}\,\|u-u_h^+\|_a^{}\sqrt{s(u_h^-,u_h^-)}\,.
\end{equation}
Alternatively, using that $a(u-u_h^{\rm FEM},u_h^-)=0$, then, we also have
\begin{equation}
s(u_h^-,u_h^-)=a(u_h^{\rm FEM}-u_h^+,u_h^-)\le\sqrt{\frac{C_{\rm equiv}^{}}{\alpha}}\|u_h^{\rm FEM}-u_h^{+}\|_a^{}\sqrt{s(u_h^-,u_h^-)}\,,
\end{equation}
which proves \eqref{error-u_h^-}.
\end{proof}

\begin{remark}[{Convergence of complementary part}]
We may interpret the last result in two ways. First,  $u_h^-$ converges to zero at least at the same speed as $u_h^+$ converges to $u$.
Moreover, \eqref{error-u_h^-} implies that in certain cases this convergence is much faster than the one for $u_h^+$. 
More precisely, focusing on the case of piecewise linear finite element functions, if the mesh satisfies the conditions
for the plain Galerkin method to admit discrete maximum principle (see, e.g., \cite{BKK08}), then $u_h^{\rm FEM}\in V_\calP^+$, which implies
that $u_h^{}=u_h^{\rm FEM}$ is also the solution to \eqref{FEM}. Thus, thanks to \eqref{error-u_h^-}, for certain meshes,
and their regular refinements, we have that $u_h^-=0$.
\end{remark}

{
\begin{remark}[Best approximation of constrained part]
In view of the best approximation result \eqref{error-u_h^+} for the constrained part $u_h^+$, there is no better finite element function with bound-preserving nodal values in the energy norm. Thus, \eqref{error-u_h^+} is the counterpart of the classical best approximation result \eqref{Galerkin-is-best-approx} for the proposed method.
\end{remark}
}

We finish this section by discussing briefly the notion of numerical solution. Since our main interest is the part of the solution
$u_h^{}$ that belongs to $V_\calP^+$, namely $u_h^+$, then the latter will be considered to be the numerical solution in the remaining
of the manuscript. The ``intermediate'' %"full"
solution $u_h^{}$ appears mostly as a tool to be able to replace the variational inequality
by an equality posed over the whole vector space $V_h^{}$.

\newcommand{\W}{\ensuremath{\Omega}}
\newcommand{\hoz}{H^{1}_0(\W)}
\newcommand{\norm}[1]{\ensuremath{\left|#1\right|}}
\newcommand{\Norm}[1]{\ensuremath{\left\|#1\right\|}}
\newcommand{\qp}[1]{\ensuremath{\!\left({#1}\right)}}
\newcommand{\bi}[2]{\ensuremath{a\qp{#1,#2}}}
\newcommand{\semi}[3]{\ensuremath{b\qp{#1;#2,#3}}}
\newcommand{\stab}[2]{\ensuremath{s\qp{#1,#2}}}
\newcommand{\qa}[1]{\ensuremath{\left\langle{#1}\right\rangle}}
\newcommand{\ltwop}[2]{\ensuremath{\qa{#1,#2}}}
\newcommand{\qnorm}[2]{\Norm{#1}_{\qp{#2}}}
\newcommand{\fes}{V_\calP^{}}

\section{A problem with nonlinear reaction}\label{Sec:nonlinear-reac}

To showcase the potential generality of the proposed bound-preserving approach, we now extend \eqref{FEM} to a
semilinear problem with monotone nonlinearity. Specifically,  for $p\in (1,\infty)$ we consider
the problem of  finding $u$ such that
\begin{equation}
  \label{eq:pde}
  \begin{split}
    -\dive\big(\calD\,\nabla u\big) + \norm{u}^{p-2} u &= f \qquad\text{ in } \W\,,
    \\
    u &= 0 \qquad\text{ on } \partial\W\,,
  \end{split}
\end{equation}
where $\calD$ satisfies the same assumptions as above.
To avoid technical diversions, we will only consider $p\ge 2$. 
This class of equation is sometimes referred to as the
Lane-Emden-Fowler equation and is related to problems with critical
exponents \cite{clement1996quasilinear}. Furthermore, they arise in
the theory of boundary layers of viscous fluids
\cite{wong1975generalized}, among other application areas.

 The weak form of this problem is given by: find $u\in \mathcal{X}:=\hoz\cap L^p(\Omega)$ such that
\begin{equation}
  \label{eq:weakform}
  \bi{u}{v} + \semi{u}{u}{v} = \ltwop{f}{v}_\Omega^{}\qquad \forall\, v\in\mathcal{X}\,,
\end{equation}
where $a(\cdot,\cdot)$ is given by \eqref{def-a} (with $\mu=0$ in this case), and the semilinear form $b$ is given by
\begin{equation}
  \label{eq:semilinear-form}
  \semi{w}{u}{v} := (\norm{w}^{p-2} u, v)_\Omega^{}\,.
\end{equation}
The space $\mathcal{X}$ is provided with the norm
\begin{equation}\label{norm-X}
\|v\|_{\mathcal{X}}^{}:= |v|_{1,\Omega}^{}+\|v\|_{0,p,\Omega}^{}\,,
\end{equation}
thus making it a reflexive Banach space. So,  using monotone operator theory (see, e.g., \cite[Chap~10]{RR04}), this problem can be proven to have a unique solution.

The error analysis of this type of problem has been carried out in several works, as early
as {\cite{Glowinski.Marrocco:74,Glowinski.Marrocco:75}} (in the context of the $p$-Laplacian). In there, the estimates are suboptimal for some values of the exponent $p$.
So, later approaches (see, e.g., \cite{BL93}) have made use of the concept of quasinorm in order to obtain
optimal error estimates. As this is the approach we will follow in this work,  we start recalling the definition of a quasinorm.

\begin{definition}[Quasinorm]
Let $V$ be a real vector space. A \emph{quasinorm}  $\qnorm{\cdot}{q} $ in $V$ is a mapping $\qnorm{\cdot}{q} :V\to\bR$
that satisfies
\begin{equation}
      \qnorm{v}{q} \geq 0 \text{ and } \qnorm{v}{q} = 0 \iff v = 0,
  \end{equation}
for all $v\in V$. 
  However, the usual triangle inequality is replaced by
  \begin{equation}
    \qnorm{v+w}{q} \leq C\qp{\qnorm{v}{q}
      +
      \qnorm{w}{q}
    },
  \end{equation}
  for all $v,w\in V$, where $C $ may depend on the definition of $\qnorm{\cdot}{q}$, and the elements $v$ and $w$ themselves.
\end{definition}

\begin{remark} 
Strictly speaking, we should also demand that the quasinorm is homogeneous, that is, $\qnorm{\alpha v}{q}= |\alpha| \qnorm{v}{q}$ for
all $v\in V$ and all $\alpha\in\mathbb{R}$. The mapping we will use to measure the error does not satisfy this last property, but
this will not affect the error estimates presented below.
\end{remark}

In our analysis below we will make use of the following quasinorm in $L^p(\Omega)$: for a given $w\in L^p(\Omega)$
we define
  \begin{equation}\label{def:quasinorm}
    \qnorm{v}{w,p}^2 := \int_\W \norm{v}^2 \qp{\norm{w} + \norm{v}}^{p-2} \textrm{d} x\,,
  \end{equation}
for all $v\in L^p(\Omega)$. {It has the following properties.}

\begin{lemma}
The mapping $\qnorm{\cdot}{w,p}^{}$ is a quasinorm in $L^p(\Omega)$. Moreover, 
  for all $v,w\in L^p(\Omega)$, the following equivalence holds
  \begin{equation}\label{eq:equivalence}
    \Norm{v}_{ 0,p,\W}^p
    \leq
    \qnorm{v}{w,p}^2
    \leq
    \||v|+|w|\|_{0,p,\Omega}^{p-2}\,
    \Norm{v}_{ 0,p,\W}^2\,.
  \end{equation}
\end{lemma}

\begin{proof}
The proof that $\qnorm{\cdot}{w,p}^{}$ is a quasinorm is completely analogous to that of \cite[Proposition~2.1]{ebmeyer2005quasi}. The equivalence \eqref{eq:equivalence}
is proven following exactly the same steps as in the proof of \cite[Equation~(10)]{ebmeyer2005quasi} (see also \cite[Proposition~3.1]{Liu2000}).
\end{proof}

In addition, the following monotonicity and continuity results can be proven for the nonlinear form $b(\cdot;\cdot,\cdot)$.

\begin{lemma}
The nonlinear form $b(\cdot;\cdot,\cdot)$ is strongly monotone with respect to the
quasinorm \eqref{def:quasinorm}. More precisely, there exists a constant $C_C^{}>0$ such that
  \begin{equation}
    \label{eq:coer}
    \semi{u}{u}{u-v}
    -
    \semi{v}{v}{u-v}
    \geq
    C_C^{}
    \qnorm{u-v}{u,p}^2.
  \end{equation}
  Moreover, for any $\theta\in (0,1]$ the following holds
  \begin{equation}
    \label{eq:bdd}
    \norm{\semi{u}{u}{w}
      -
      \semi{v}{v}{w}}
    \leq
    C_B^{}
    \qp{\theta\qnorm{u-v}{u,p}^2
      +
      \theta^{{-p+1}}  \qnorm{w}{u,p}^2
    }\qquad\forall\, u,v,w\in L^p(\Omega)\,,
  \end{equation}
  where $C_B^{}$ depends only on $p$.
\end{lemma}

\begin{proof} In \cite[Lemma~2.1]{BL93} the following bounds are proven
\begin{align}
\big( |x|^{p-2}x-|y|^{p-2}y\big)\, (x-y) &\ge C_2^{}\big(|x|+|y|\big)^{p-2}|x-y|^2\,,\label{BL:bound-1}\\
\big| |x|^{p-2}x-|y|^{p-2}y\big| &\le C_1^{}\big(|x|+|y|\big)^{p-2}|x-y|\,, \label{BL:bound-2}
\end{align}
for all $x,y\in\mathbb{R}$. Next, \eqref{eq:coer}  follows by from \eqref{BL:bound-1} and noticing that 
$|x|+|y|\ge \big( |x|+|x-y|\big)/2$.
To prove \eqref{eq:bdd} we combine \eqref{BL:bound-2} with the result from \cite[Lemma~2.2]{Liu2000}.
\end{proof}

\subsection{The finite element method}
The finite element method we consider is the natural extension of
\eqref{FEM}, that is: find $u_h^{}\in \fes$ such that
\begin{equation}
  \label{eq:NLFEM}
  \bi{u_h^+}{v_h^{}}
  +
  \semi{u_h^+}{u_h^+}{v_h^{}}
  +
  \stab{u_h^-}{v_h^{}}
  =
  \ltwop{f}{v_h^{}}_\Omega^{}\qquad\forall\, v_h^{}\in V_\calP^{}\,.
\end{equation}
Here, the stabilisation term is given by
\begin{equation}
s(v_h^{},w_h^{}):= \alpha\sum_{i=1}^N \|\calD\|_{0,\infty,\omega_i}^{}\frakh(\bx_i^{})^{d-2}\,v_h^{}(\bx_i^{})w_h^{}(\bx_i^{})\,, \label{def-s-nonlinear}
%\alpha\,\|\calD\|_{0,\infty,\Omega}^{} h^{-2}\,\big(v_h^{},w_h^{})_h^{}\,, \label{def-s-nonlinear}
\end{equation}
where, once again, 
$\alpha>0$ is an non-dimensional constant. %, and $(\cdot,\cdot)_h^{}$ is the lumped $L^2(\Omega)$-inner product defined in \eqref{lumped-product}. 
Following very similar steps to those from the proof
of Theorem~\ref{Theo:Well-posedness},  Method~\eqref{eq:NLFEM} can be proven to have a unique solution $u_h^{}\in \fes$. We note that there appears to be no traceable numerical advantage in including a linearised reaction term in the stabilisation \eqref{def-s-nonlinear}, at least for modestly large values of $p$, so we omit it for simplicity. 

As in the linear case, \eqref{eq:NLFEM} can be linked to a variational inequality, as the following result
(whose proof is totally analogous to that of  Theorem~\ref{Th:Variational-Inequality}) shows.

\begin{theorem}
  \label{the:nl-varin}
  Let $u_h^{}$ be the unique solution of (\ref{eq:NLFEM}). Then, $u_h^+$
  satisfies the variational inequality
  \begin{equation}
    \bi{u_h^+}{v_h^{} - u_h^+}
    +
    \semi{u_h^+}{u_h^+}{v_h^{} - u_h^+}
    \geq
    \ltwop{f}{v_h^{} - u_h^+}_\Omega^{}\qquad
    \forall\, v_h^{} \in V_\calP^+.
  \end{equation}
\end{theorem}

The following result is the main reason for the use of a quasinorm instead of the norm induced by
the problem. In fact, starting from the last result a C\'ea type estimate can be obtained, but, analogously
to what is reported in \cite{ciarlet2002finite}, that would lead to a suboptimal estimate for certain values of $p$. So,
in the next result we provide an optimal estimate with respect to the quasinorm \eqref{def:quasinorm}.

\begin{theorem}
  \label{the:best-approx}
  Let $u$ solve (\ref{eq:pde}) and $u_h^{} \in\fes$ be the  solution of \eqref{eq:NLFEM}. Then
   \begin{equation}
    \Norm{u - u_h^{+}}_a^2 + \qnorm{u - u_h^{+}}{u,p}^2
    \leq
    C
    \inf_{v_h^{}\in V_\calP^+}
    \qp{
      \Norm{u - v_h^{}}_a^2 + \qnorm{u - v_h^{}}{u,p}^2
    }.
  \end{equation}
\end{theorem}

\begin{proof}
  Using the coercivity  \eqref{eq:coer} of $b$ followed by  Theorem~\ref{the:nl-varin} we get
  \begin{equation}
    \begin{split}
      C_C^{}\qp{
        \Norm{u - u_h^+}_a^2
        +
        \qnorm{u - u_h^+}{u,p}^2
      }
      &\leq
      \bi{u - u_h^+}{u - u_h^+}
      +
      \semi{u}{u}{u-u_h^+}
      -
      \semi{u_h^+}{u_h^+}{u-u_h^+}
      \\
      &\leq
      \bi{u - u_h^+}{u - v_h^{}}
      +
      \semi{u}{u}{u-v_h^{}}
      -
      \semi{u_h^+}{u_h^+}{u-v_h^{}},
    \end{split}
  \end{equation}
  for every $v_h^{}\in V_\calP^+$. Next, using Young's inequality and 
  \eqref{eq:bdd} we have, for all $\theta\in (0,1]$,
  \begin{equation}
    \begin{split}
      C_C^{}
      \qp{
        \Norm{u - u_h^{+}}_a^2
        +
        \qnorm{u - u_h^{+}}{u,p}^2
      }
      &\leq
      \frac{C_C^{}}4 \Norm{u - u_h^+}_a^2
      +
      C_C^{} \Norm{u - v_h^{}}_a^2
      \\
      &\qquad +
      C_B^{}
      \qp{\theta \qnorm{u-u_h^+}{u,p}^2
        +
        \theta^{\textcolor{violet}{-p+1}} \qnorm{u-v_h^{}}{u,p}^2
      }.
    \end{split}
  \end{equation}
  Choosing $\theta = \min\qp{\frac{C_C^{}}{4C_B^{}}, \frac 1 2}$ and rearranging the inequality yields the desired result. 
\end{proof}

\section{Numerical tests}\label{Sec:Numer}

In these tests we detail some aspects of our implementation and
showcase the methodology looking at the symmetric problem
(\ref{reaction-diffusion}) and its nonlinear counterpart
(\ref{eq:pde}) in 2d. We recall once again that all references
to the {\sl numerical solution} refer to the function $u_h^+\in V_\calP^+$, and
not the function $u_h^{}$.

To linearise the problem, we pose the following Richardson-like
iterative approximation for (\ref{FEM}): Given $u^0$ and $\omega\in
(0,1]$, for each $n=0, 1, \dots$ find $u^{n+1}$ such that
\begin{equation}
  \label{eq:iterative}
  a(u^{n+1}, v)
  =
  a(u^{n}, v) + \omega \qp{\langle f, v \rangle_\Omega^{} - a(\qp{u^n}^+, v) - s(\qp{u^n}^-, v)}.
\end{equation}

We initialise the finite element approximation of (\ref{eq:iterative})
{by the} Galerkin approximation, that is we set $u^0_h \in \fes$
such that, for all $v_h \in \fes$
\begin{equation}
  a(u^0_h, v_h^{}) = \langle f, v_h^{}\rangle_\W^{}.
\end{equation}
Then, the approximation to the iteration (\ref{eq:iterative}) becomes
for each $n=0,\dots, N-1$ find $u^{n+1}_h \in \fes$ such that for all
$v_h^{} \in \fes$
\begin{equation}
  \label{eq:iterativeFEM}
  a(u^{n+1}_h, v_h^{})
  =
  a(u^{n}_h, v_h^{})
  +
  \omega \qp{\langle f, v_h^{}\rangle_\W^{} - a(\qp{u^n_h}^+, v_h^{}) - s(\qp{u^n_h}^-, v_h^{})}.
\end{equation}
In each experiment we take $\alpha = 1$ within the stabilisation. The linear systems 
arising in \eqref{eq:iterativeFEM} are solved using an LU decomposition within the Eigen
library. The linearisation was terminated when $\Norm{u_h^{n+1} -
  u_h^n}_{0,\W} \leq 10^{-12}$.

\subsection{Convergence on a regular grid with a smooth solution}

We first consider $\calD = \epsilon \calI$,  {where $\calI$ denotes the $2\times 2$ identity matrix}, 
with $\epsilon = 10^{-5}$, $\mu = 1$,
and set $f$ such that the function
\begin{equation}
  u(x,y) = \sin(\pi x) \sin(\pi y)
\end{equation}
solves the problem
(\ref{reaction-diffusion}) over the unit square. 

Convergence results for piecewise linear, $k=1$, and quadratic, $k=2$,
elements over a sequence of uniformily refined criss-cross meshes are
shown in Figures \ref{fig:sin1} and \ref{fig:sin2} respectively.  In
line with the error estimates from Section~\ref{Sec:error}, the method
converges with optimal rate in the function approximation sense for
both $k=1$ and $k=2$. Notice also that the number of iterations of the
linearisation decreases as a function of $h$. {Here,
  and thereafter, EOC stands for \textit{estimated order of
    convergence}.}

\begin{figure*}[h!]
    \centering
    \begin{subfigure}[t]{0.5\textwidth}
        \centering
        \includegraphics[width=1.\textwidth]{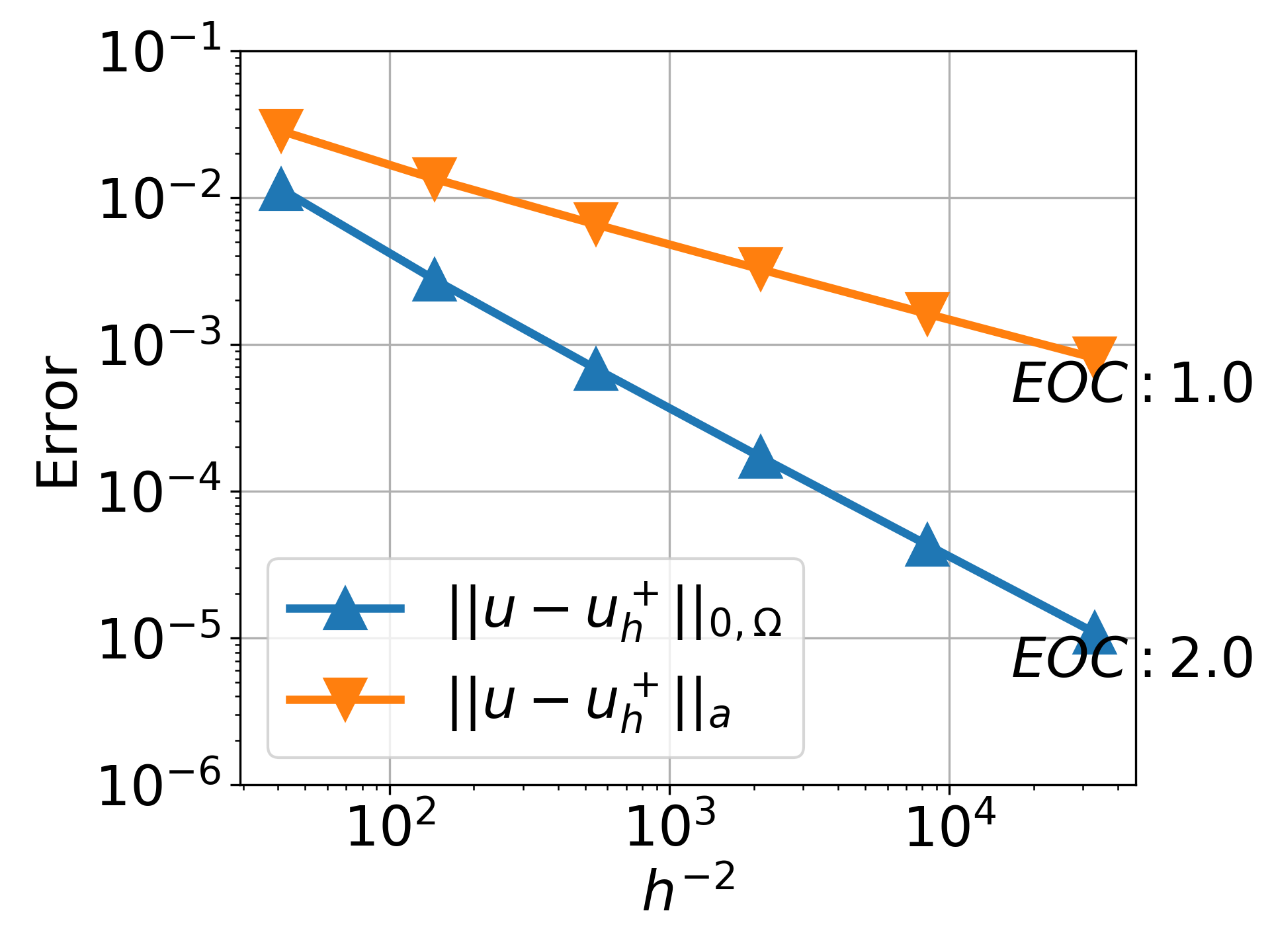}
        \caption{Convergence in the $L^2$ and energy norms.}
    \end{subfigure}%
    ~
    \begin{subfigure}[t]{0.5\textwidth}
        \centering
        \includegraphics[width=1.\textwidth]{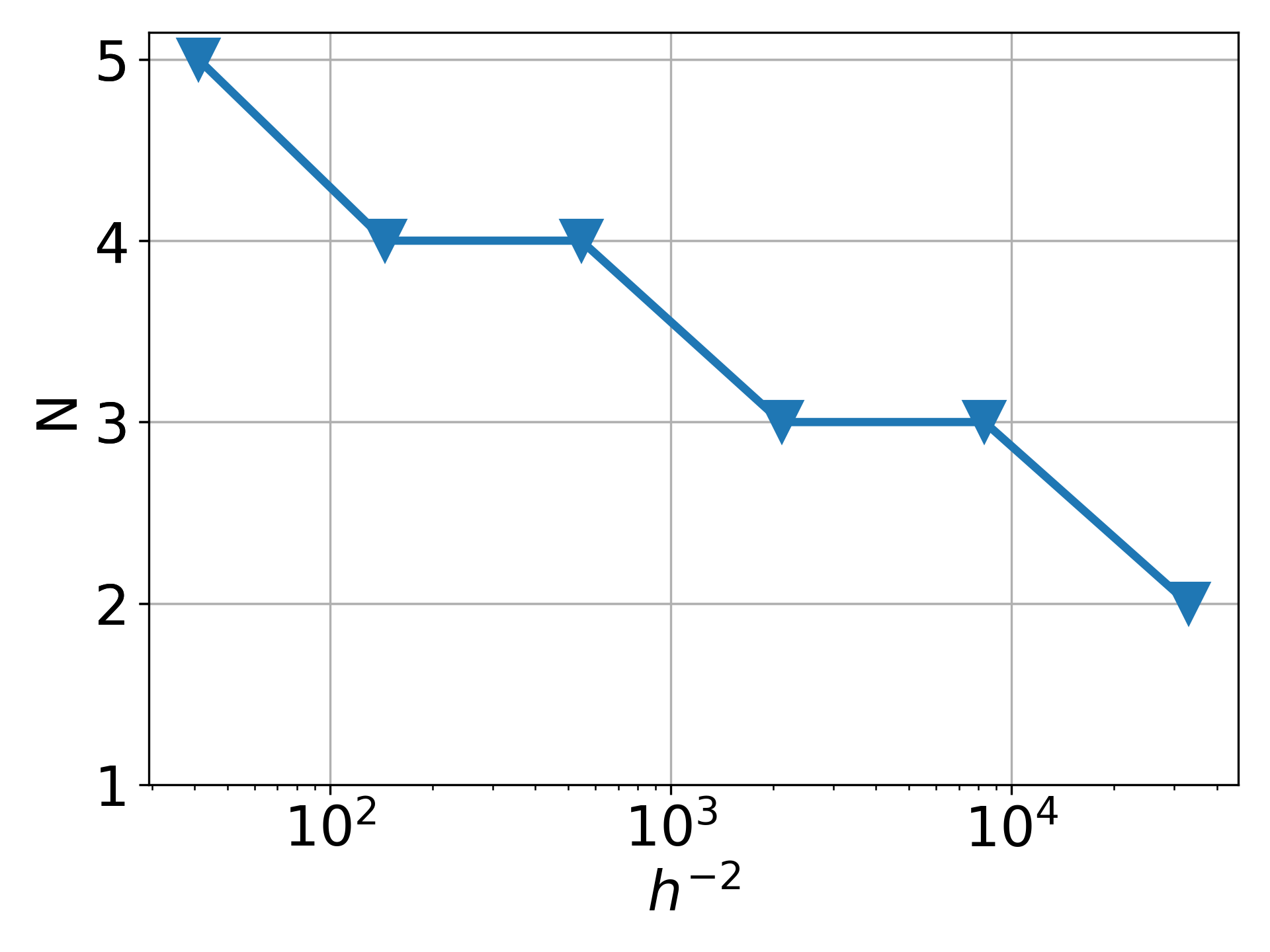}
        \caption{Number of iterations to achieve convergence as a
          function of the meshsize.}
    \end{subfigure}
    \caption{\label{fig:sin1} We test the convergence of a piecewise
      linear conforming approximation given by (\ref{eq:iterativeFEM})
      on a sequence of concurrently refined criss-cross meshes. The
      Richardson iteration converges for $\omega=1$. }
\end{figure*}

\begin{figure*}[h!]
    \centering
    \begin{subfigure}[t]{0.5\textwidth}
        \centering
        \includegraphics[width=1.\textwidth]{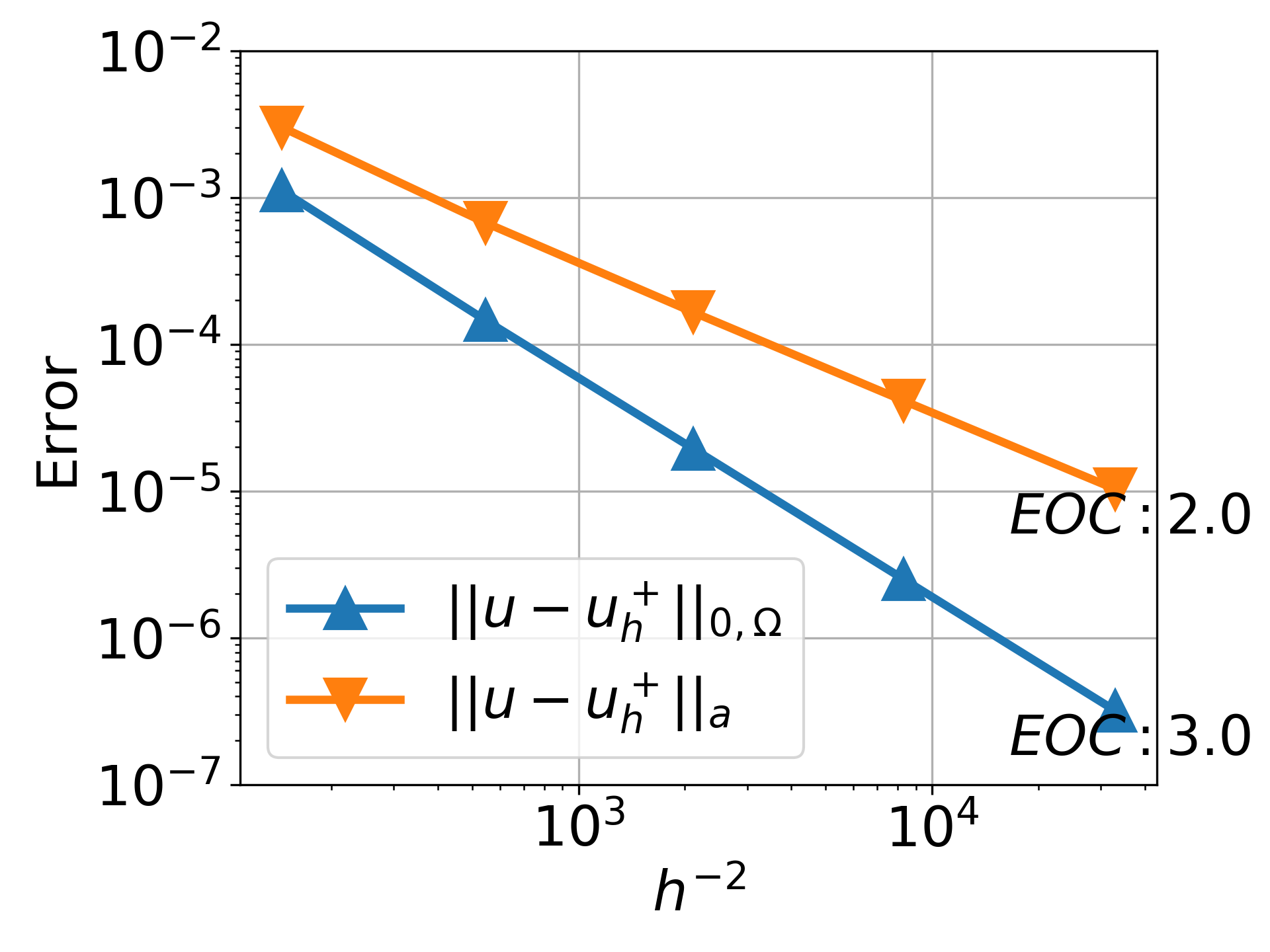}
        \caption{Convergence in the $L^2$ and energy norms.}
    \end{subfigure}%
    ~
    \begin{subfigure}[t]{0.5\textwidth}
        \centering
        \includegraphics[width=1.\textwidth]{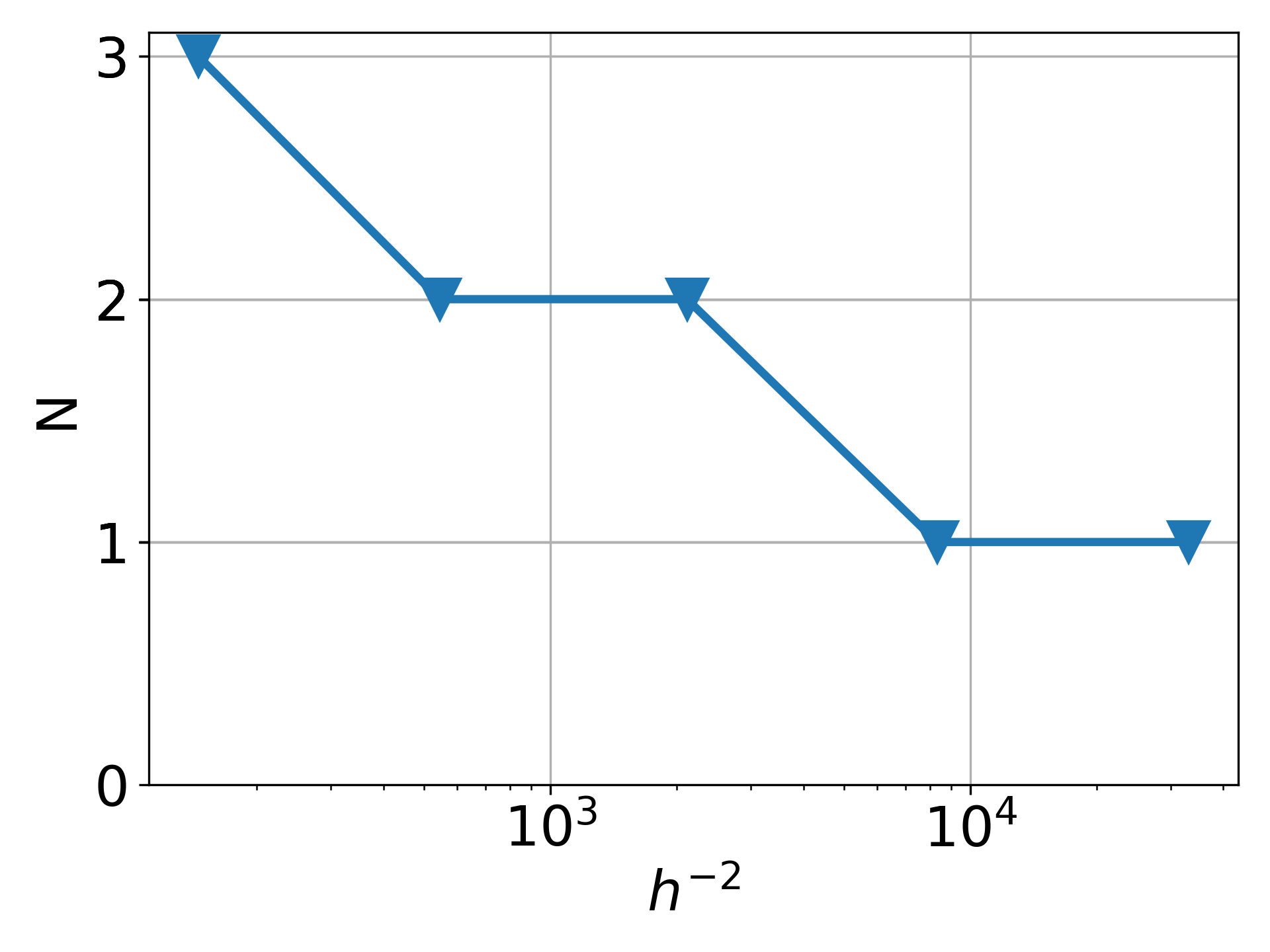}
        \caption{Number of iterations to achieve convergence as a
          function of the meshsize.}
    \end{subfigure}
    \caption{\label{fig:sin2} We test convergence of a piecewise
      quadratic conforming approximation (\ref{eq:iterativeFEM}) on a
      sequence of concurrently refined criss-cross meshes. The
      Richardson iteration converges for $\omega=1$.}
\end{figure*}

\subsection{Convergence on an obtuse grid with a smooth solution}

We again consider $\calD = \epsilon \calI$, with $\epsilon = 10^{-5}$, $\mu = 1$.
The analytical solution of \eqref{reaction-diffusion} over $(-1,1)\times (0,1)$  {is taken as}
\begin{equation}
  u(x,y) = \sin(\pi (x+1)/2) \sin(\pi y)\,,
\end{equation}
and  we compute $f$ accordingly. Notice that $u(\bx)\in[0,1]$ for all $\bx\in\Omega$.
We pose the problem over a triangulation
with obtuse elements as described in \cite{BKK09} illustrated in
Figure \ref{fig:obtuse-refinements}. This  was used in \cite{BKK09}
as an example of triangulations for which the finite element method does not satisfy the
discrete maximum principle, even for the Poisson equation; so, it  poses a
challenge to the finite element method as solutions do not in general
satisfy DMP even if $\epsilon \gg 1$.

Convergence results for piecewise linear elements over this mesh are
given in Figure \ref{fig:obtuse-conv}. The method converges optimally
and similar results are observed for higher order elements. Notice
that the linearisation takes more iterations to achieve convergence, 
which is to be expected, as the Galerkin solution will, very likely, never 
respect the bounds for the problem, regardless of how fine the mesh is
(the iteration count remains, nevertheless, low).

\begin{figure*}[h!]
  \centering
    \begin{subfigure}[t]{0.3\textwidth}
        \centering
        \includegraphics[width=1.\textwidth]{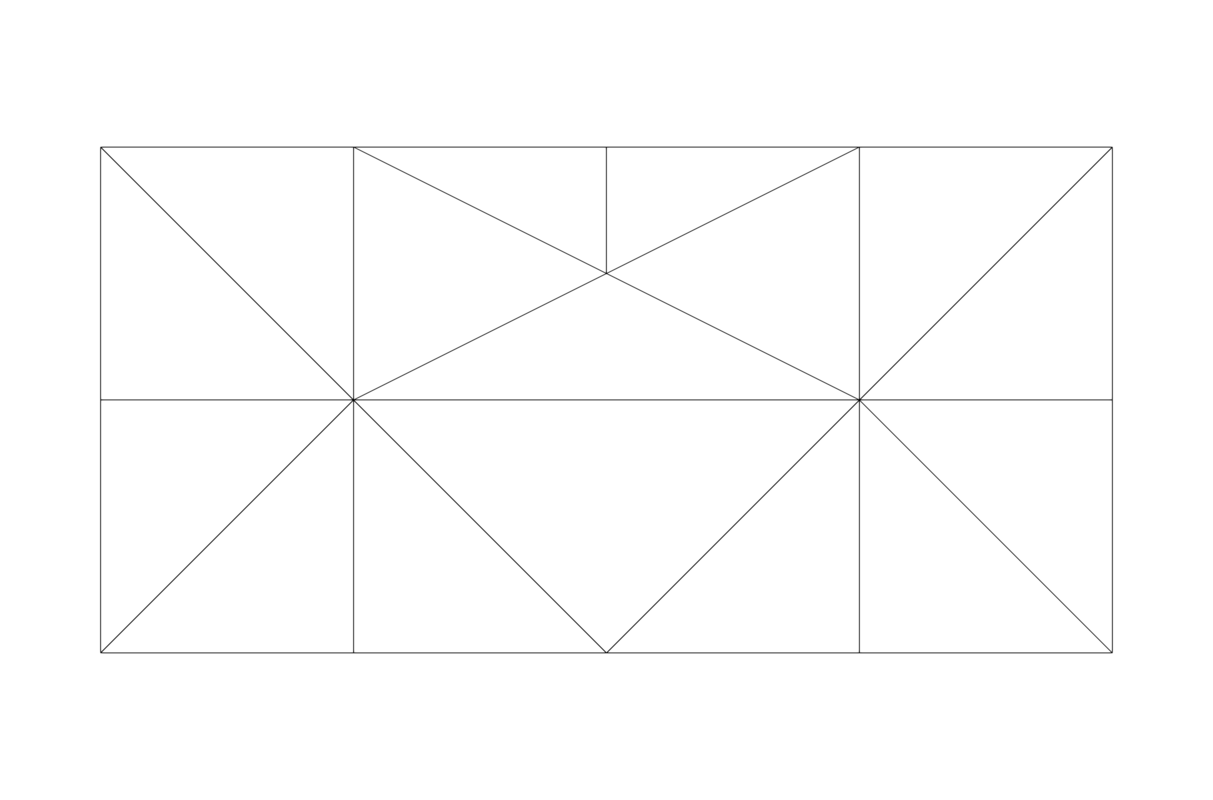}
        \caption{Iteration 1.}
    \end{subfigure}%
    ~
    \begin{subfigure}[t]{0.3\textwidth}
        \centering
        \includegraphics[width=1.\textwidth]{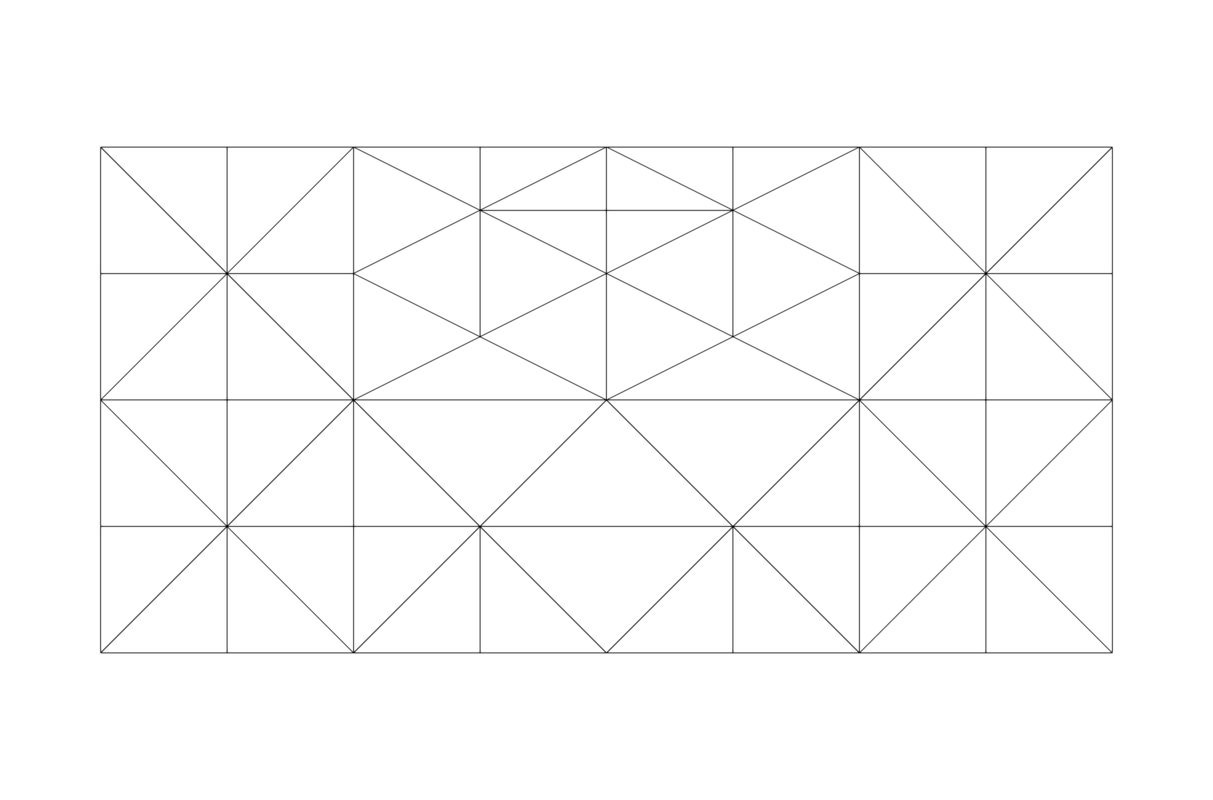}
        \caption{Iteration 2.}
    \end{subfigure}
    \begin{subfigure}[t]{0.3\textwidth}
        \centering
        \includegraphics[width=1.\textwidth]{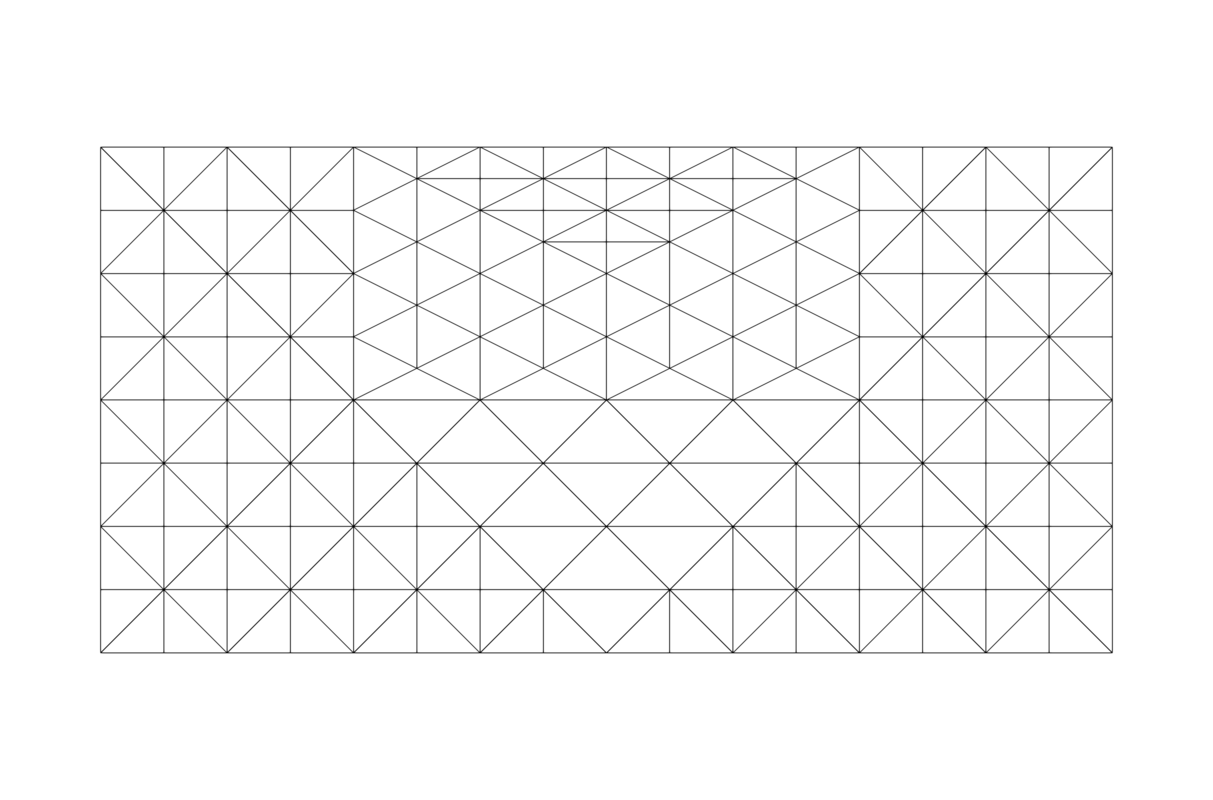}
        \caption{Iteration 3.}
    \end{subfigure}
    \caption{
      \label{fig:obtuse-refinements}
      Three refinements of the mesh with a single obtuse
      triangle from \cite{BKK09}. The result is a layer of obtuse triangles.}
\end{figure*}

\begin{figure*}[h!]
    \centering
    \begin{subfigure}[t]{0.5\textwidth}
        \centering
        \includegraphics[width=1.\textwidth]{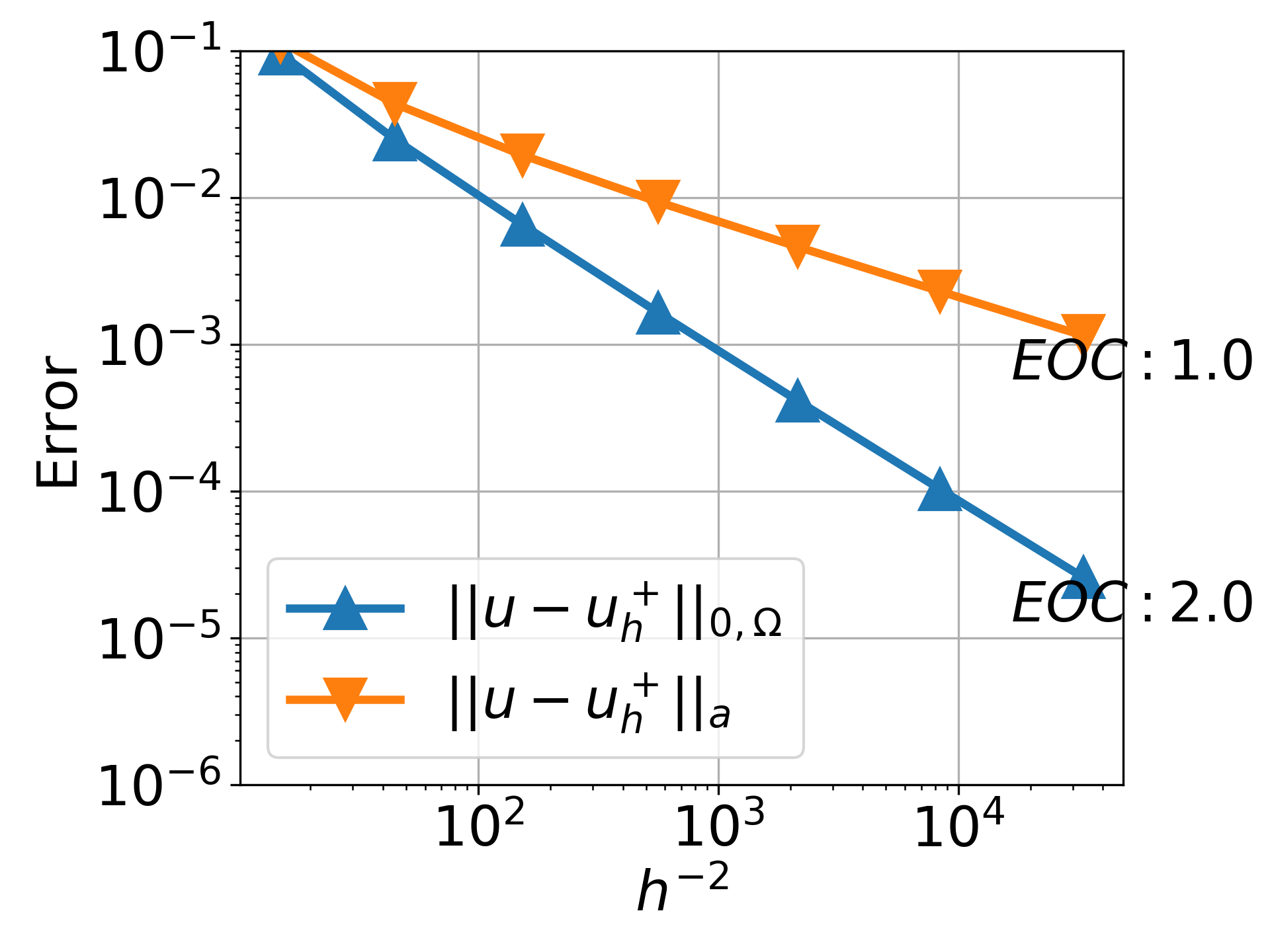}
        \caption{Convergence in the $L^2$ and energy norms.}
    \end{subfigure}%
    ~
    \begin{subfigure}[t]{0.5\textwidth}
        \centering
        \includegraphics[width=1.\textwidth]{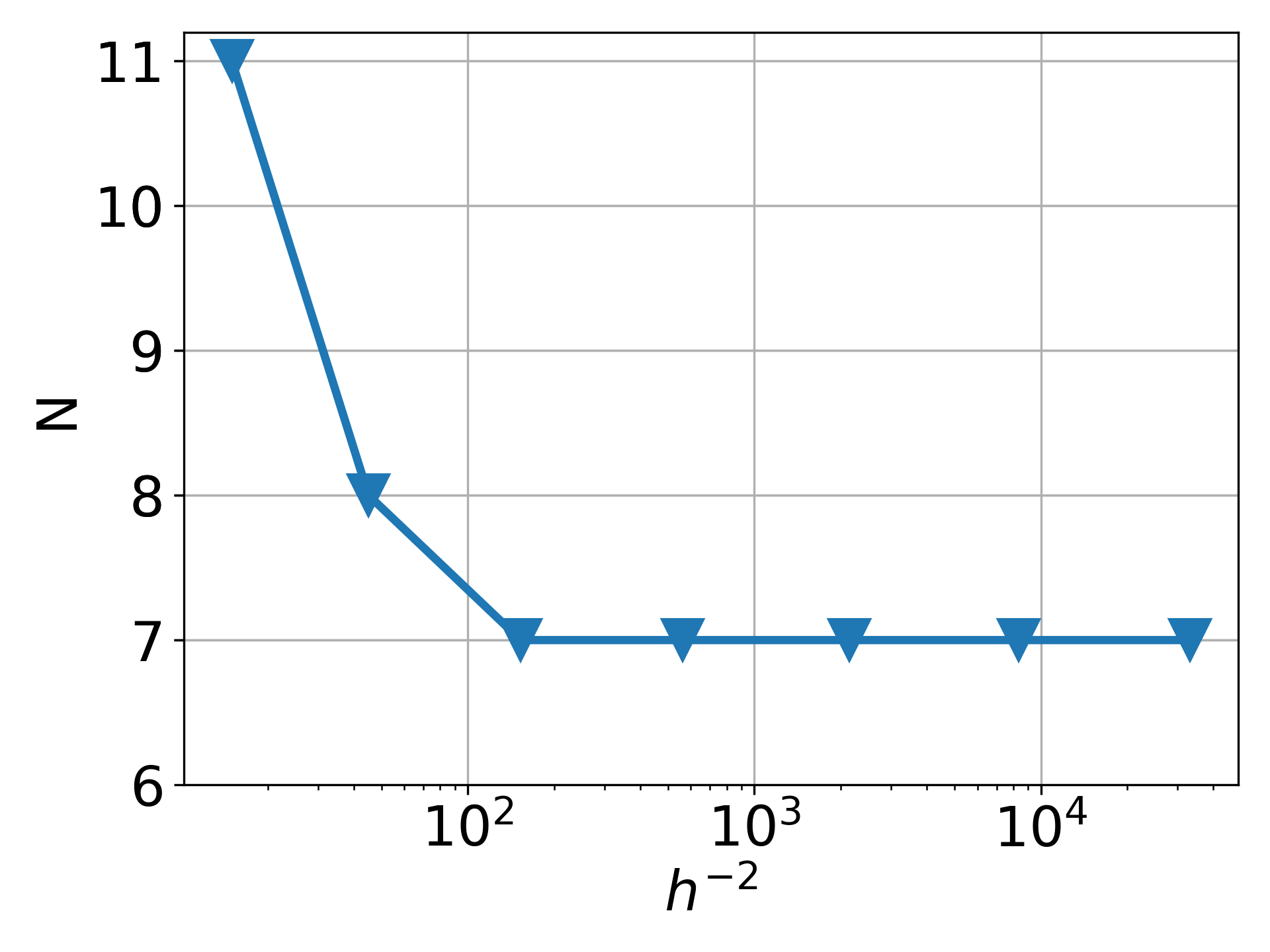}
        \caption{Number of iterations to achieve convergence as a
          function of the meshsize.}
    \end{subfigure}
    \caption{\label{fig:obtuse-conv} We test convergence of a
      piecewise linear conforming approximation
      (\ref{eq:iterativeFEM}) on a sequence of concurrently refined
      meshes with obtuse elements as illustrated in Figure
      \ref{fig:obtuse-refinements}. Here the Richardson linearisation
      converges for $\omega =1$.}
\end{figure*}

\subsection{Resolution of boundary layers}

Consider the problem
\begin{equation}
  \label{eq:layers}
  \begin{split}
    -\epsilon\Delta u + u &= 1 \text{ in } \Omega\,,
    \\
    u &= 0 \text{ on } \partial \Omega.
  \end{split}
\end{equation}
We fix $h\approx 0.02$ on a criss-cross and vary $\epsilon \in [10^{-2}, 10^{-7}]$. For
particularly small $\epsilon$ the Richardson iteration required
dampening for convergence. With $\epsilon > 10^{-5}$, we use $\omega =
1$ and convergence was achieved within 4 iterations. When $\epsilon
\leq 10^{-5}$, $\omega = \tfrac 12$ is sufficient for convergence
with fewer than 46 iterations in each case. The most challenging case
being the smallest value of $\epsilon$. Computed solutions for
different values of $\epsilon$ are shown in Figure
\ref{fig:solutions-hom-dir}.

\begin{figure*}[h!]
    \centering
    \begin{subfigure}[t]{0.3\textwidth}
        \centering
        \includegraphics[width=1.\textwidth]{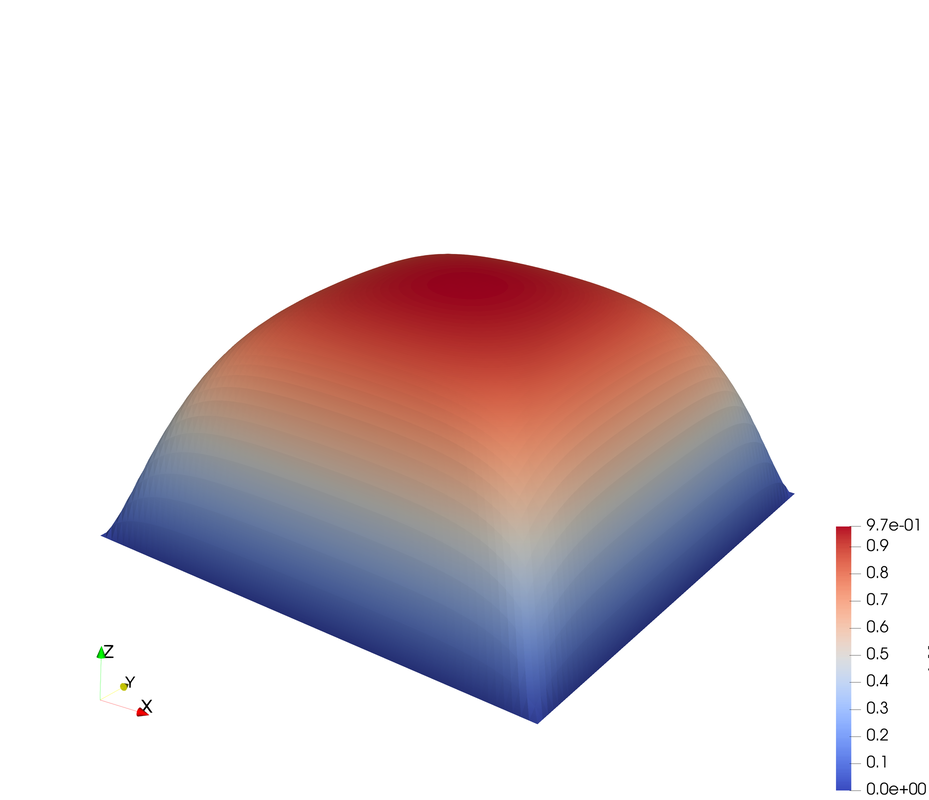}
        \caption{$\epsilon = 10^{-2}$.}
    \end{subfigure}%
    ~
    \begin{subfigure}[t]{0.3\textwidth}
        \centering
        \includegraphics[width=1.\textwidth]{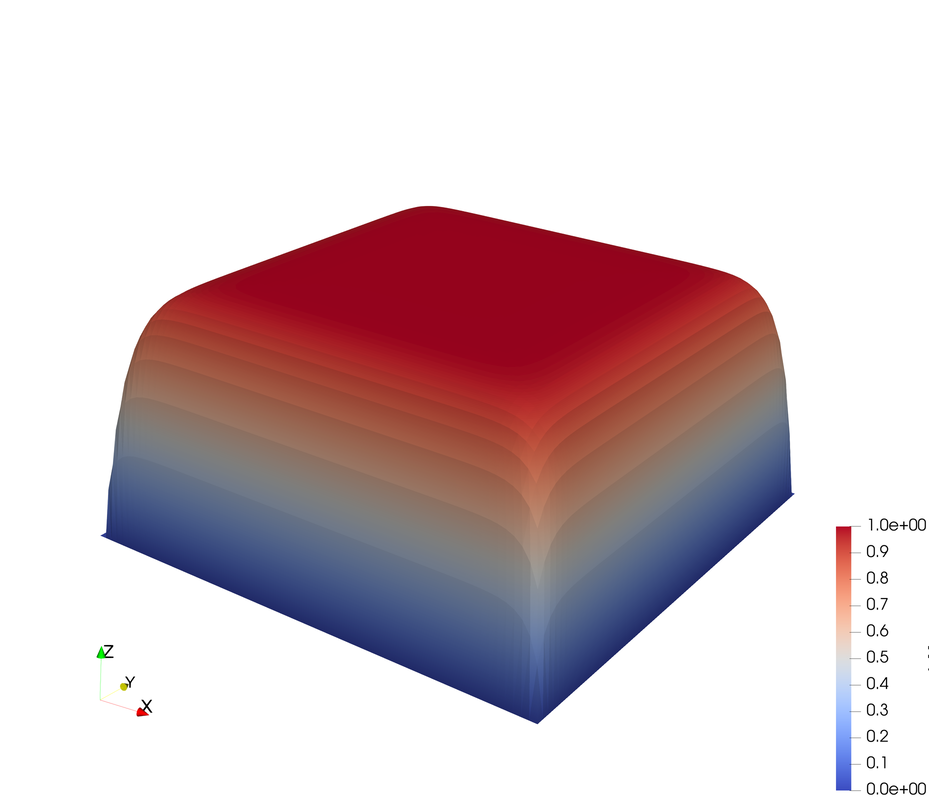}
        \caption{$\epsilon = 10^{-3}$.}
    \end{subfigure}%
    ~
    \begin{subfigure}[t]{0.3\textwidth}
        \centering
        \includegraphics[width=1.\textwidth]{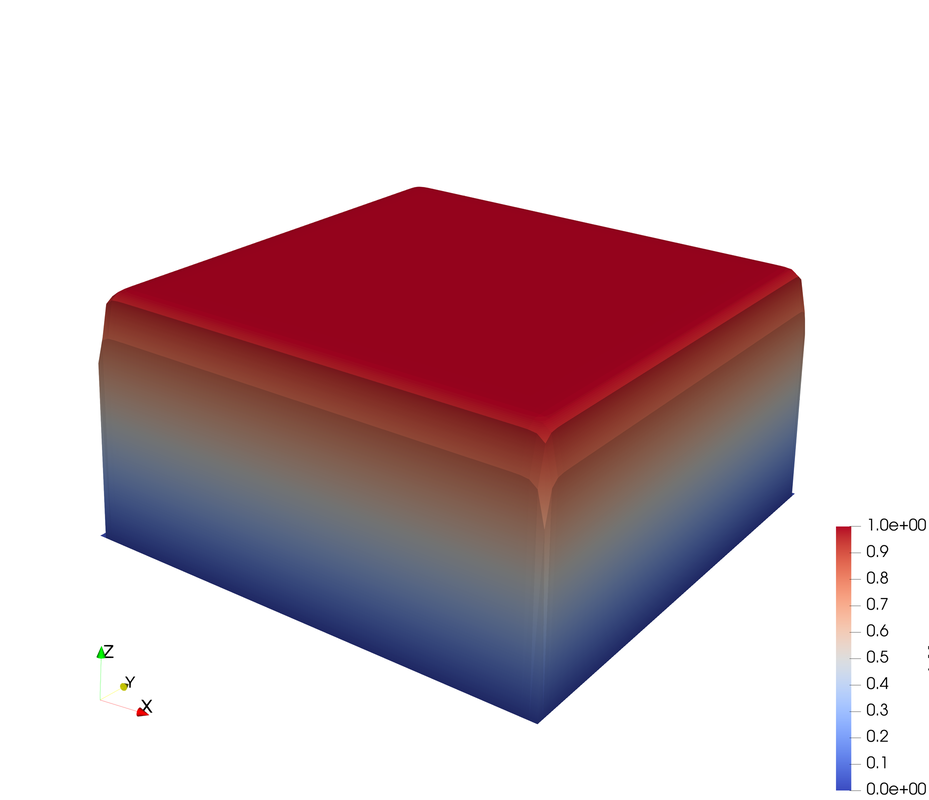}
        \caption{$\epsilon = 10^{-4}$.}
    \end{subfigure}%
    \\
    \begin{subfigure}[t]{0.3\textwidth}
        \centering
        \includegraphics[width=1.\textwidth]{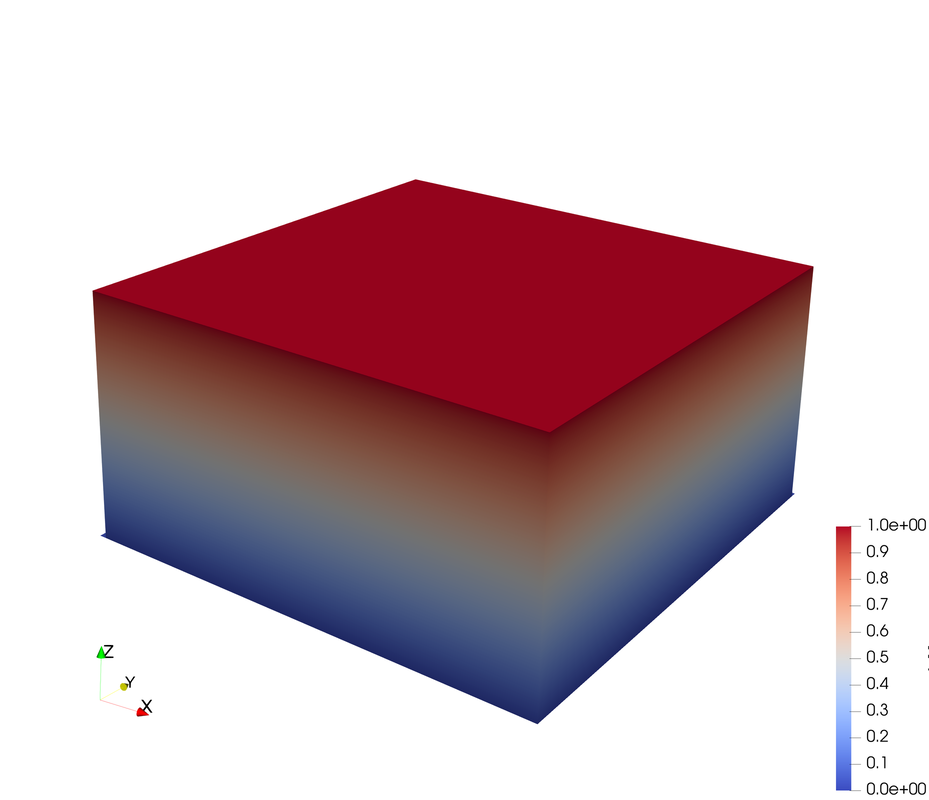}
        \caption{$\epsilon = 10^{-5}$.}
    \end{subfigure}%
    ~
    \begin{subfigure}[t]{0.3\textwidth}
        \centering
        \includegraphics[width=1.\textwidth]{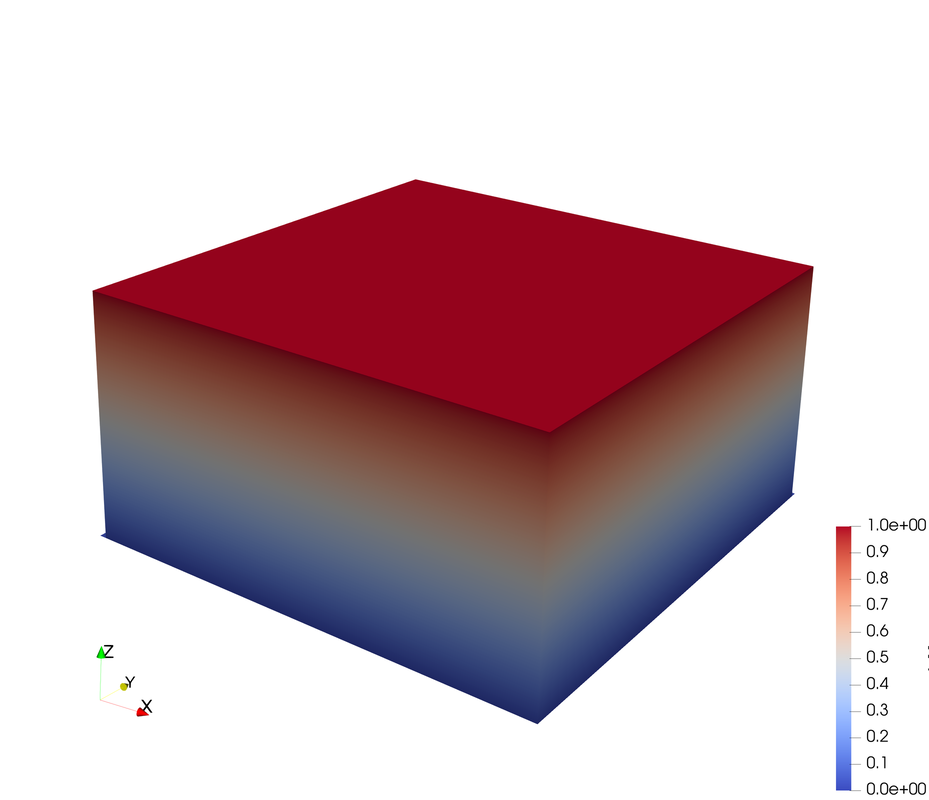}
        \caption{$\epsilon = 10^{-6}$.}
    \end{subfigure}%
    ~
    \begin{subfigure}[t]{0.3\textwidth}
        \centering
        \includegraphics[width=1.\textwidth]{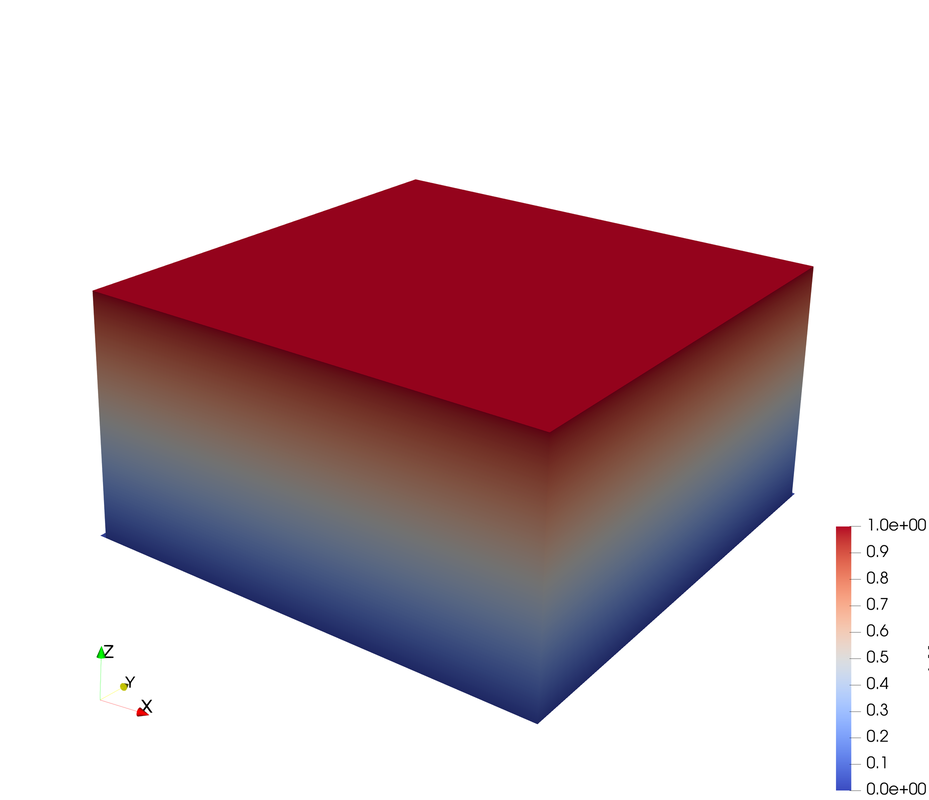}
        \caption{$\epsilon = 10^{-7}$.}
    \end{subfigure}%
    \caption{\label{fig:solutions-hom-dir} Elevations of the
      approximation to (\ref{eq:layers}) for fixed $h$ and different
      values of $\epsilon$. We notice the absence of oscillations %there are no apparent oscillations
      even for particularly small values of $\epsilon$.}
\end{figure*}

\subsection{Resolution of boundary layers with discontinuous Dirichlet conditions}

Consider the problem
\begin{equation}
  \label{eq:discbc}
  \begin{split}
    -\epsilon\Delta u + u &= 0 \text{ in } \Omega\,,
    \\
    u &= g_D^{}
    \text{ on } \partial \Omega,
  \end{split}
\end{equation}
where $g_D^{} = 1$ on $[0,\tfrac 12]\times \{0\}$, $g_D^{} = 0$ on
$(\tfrac 12, 1]\times \{0\}$ and periodically follows the same pattern
  counter-clockwise. We fix $h\approx 0.02$ on a criss-cross mesh and vary
  $\epsilon \in [10^{-2}, 10^{-7}]$. For particularly small $\epsilon$
  the Richardson iteration required dampening for convergence. With
  $\epsilon > 10^{-5}$ we used $\omega = 1$ and convergence was
  achieved within 5 iterations. {When $\epsilon \leq 10^{-5}$, then} $\omega =
  0.5$ provided a convergent algorithm and took fewer than 40
  iterations in each case. Figure \ref{fig:solutions-non-hom-dir}
  shows some solutions.

\begin{figure*}[h!]
    \centering
    \begin{subfigure}[t]{0.3\textwidth}
        \centering
        \includegraphics[width=1.\textwidth]{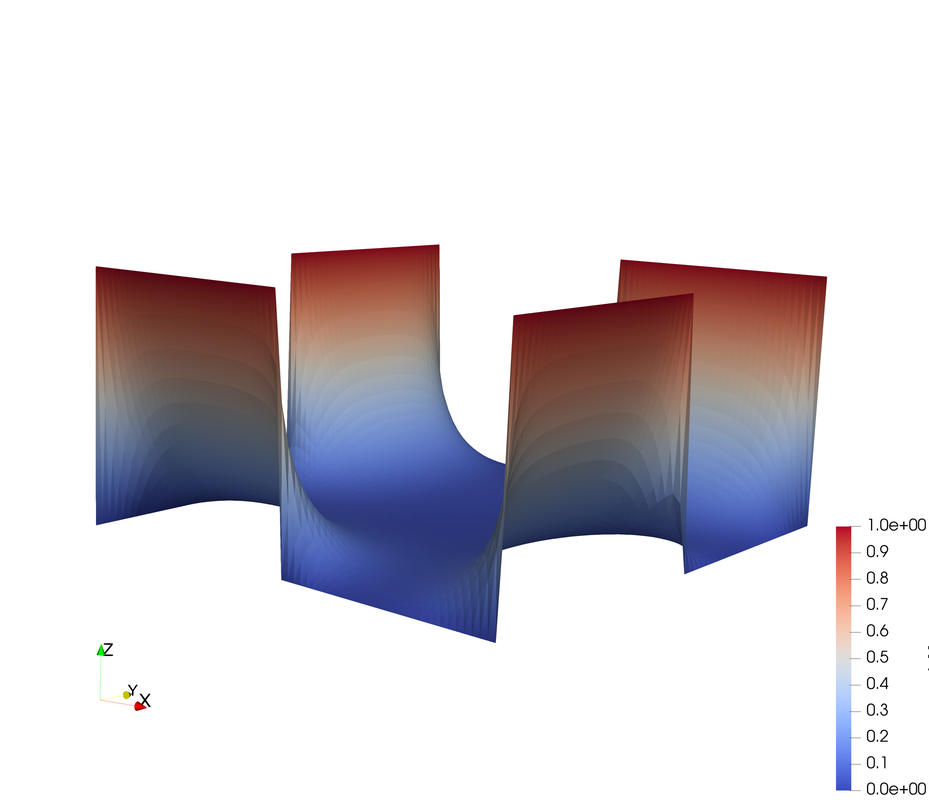}
        \caption{$\epsilon = 10^{-2}$.}
    \end{subfigure}%
    ~
    \begin{subfigure}[t]{0.3\textwidth}
        \centering
        \includegraphics[width=1.\textwidth]{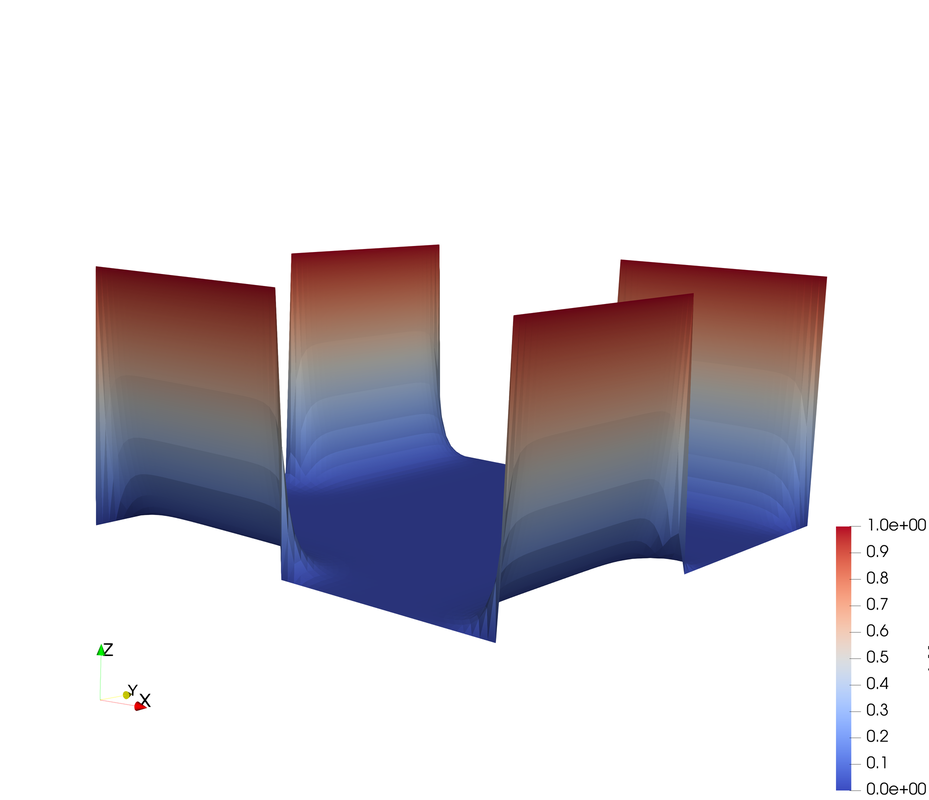}
        \caption{$\epsilon = 10^{-3}$.}
    \end{subfigure}%
    ~
    \begin{subfigure}[t]{0.3\textwidth}
        \centering
        \includegraphics[width=1.\textwidth]{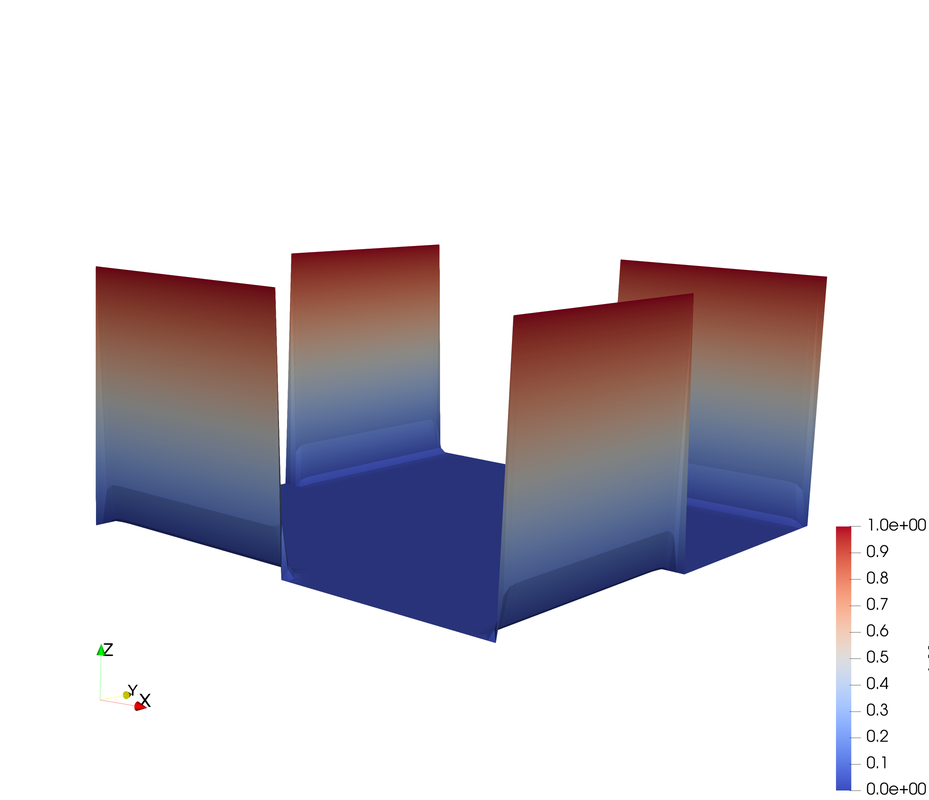}
        \caption{$\epsilon = 10^{-4}$.}
    \end{subfigure}%
    \\
    \begin{subfigure}[t]{0.3\textwidth}
        \centering
        \includegraphics[width=1.\textwidth]{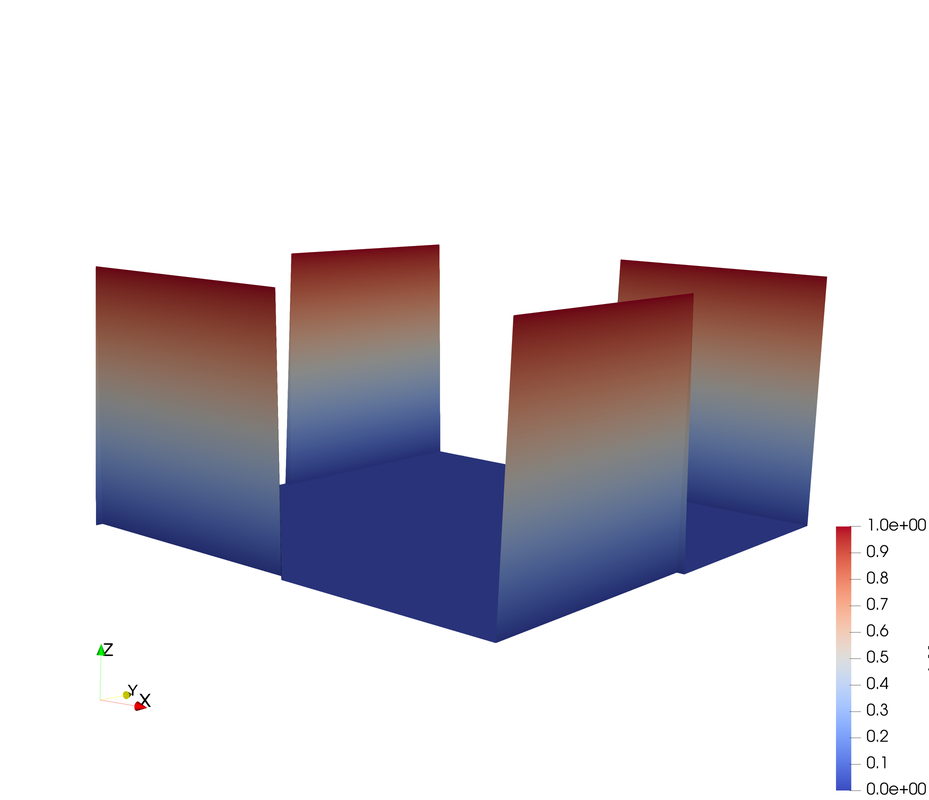}
        \caption{$\epsilon = 10^{-5}$.}
    \end{subfigure}%
    ~
    \begin{subfigure}[t]{0.3\textwidth}
        \centering
        \includegraphics[width=1.\textwidth]{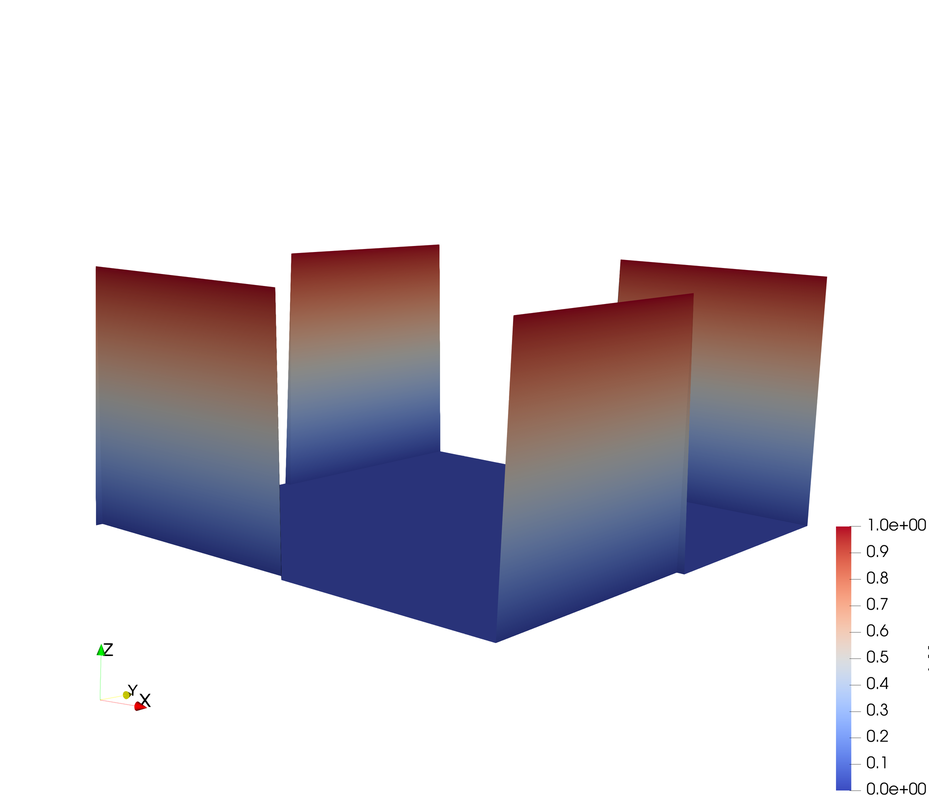}
        \caption{$\epsilon = 10^{-6}$.}
    \end{subfigure}%
    ~
    \begin{subfigure}[t]{0.3\textwidth}
        \centering
        \includegraphics[width=1.\textwidth]{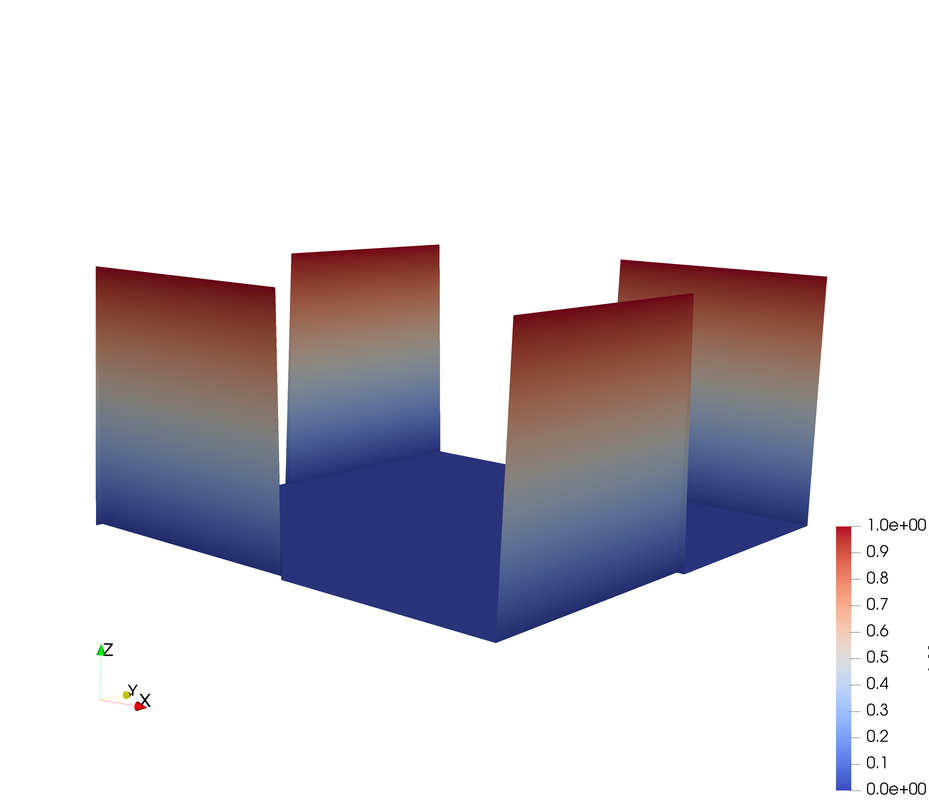}
        \caption{$\epsilon = 10^{-7}$.}
    \end{subfigure}%
    \caption{\label{fig:solutions-non-hom-dir} Elevations of the
      approximation to (\ref{eq:discbc}) for fixed $h$ and varying
      $\epsilon$. We notice there are no apparent oscillations even for
      particularly small values of $\epsilon$.}
\end{figure*}

\subsection{A solution with an interior layer}

Consider the problem
\begin{equation}
  \begin{split}
    -\epsilon \Delta u + u &= f \text{ in } \W\,,
    \\
    u &= 0 \text{ on } \partial \W,
  \end{split}
\end{equation}
with
\begin{equation}
  f =
  \begin{cases}
    \frac 12 \text{ in } [\tfrac 14, \tfrac 34]^2  
    \\
    1 \text{ otherwise.}
  \end{cases}
\end{equation}
In this case the solution is expected to achieve a local minimum on
the interior. We fix $h \approx 0.02$ and examine the solution for
$\epsilon = \{10^{-4}, 10^{-7}\}.$ We compare the standard finite element
solution and the approximation given by (\ref{eq:iterativeFEM}) in
Figure \ref{fig:intmin}. Notice that the  plain Galerkin solution has
oscillations near the boundary layer that become extreme for $\epsilon
\ll 1$, which are totally removed by the current method. In addition, for $k=2$, there
are noticeable undershoots around the interior layer, which are totally removed by
the current method.

\begin{figure*}[h!]
    \centering
    \begin{subfigure}[t]{0.3\textwidth}
        \includegraphics[width=\textwidth]{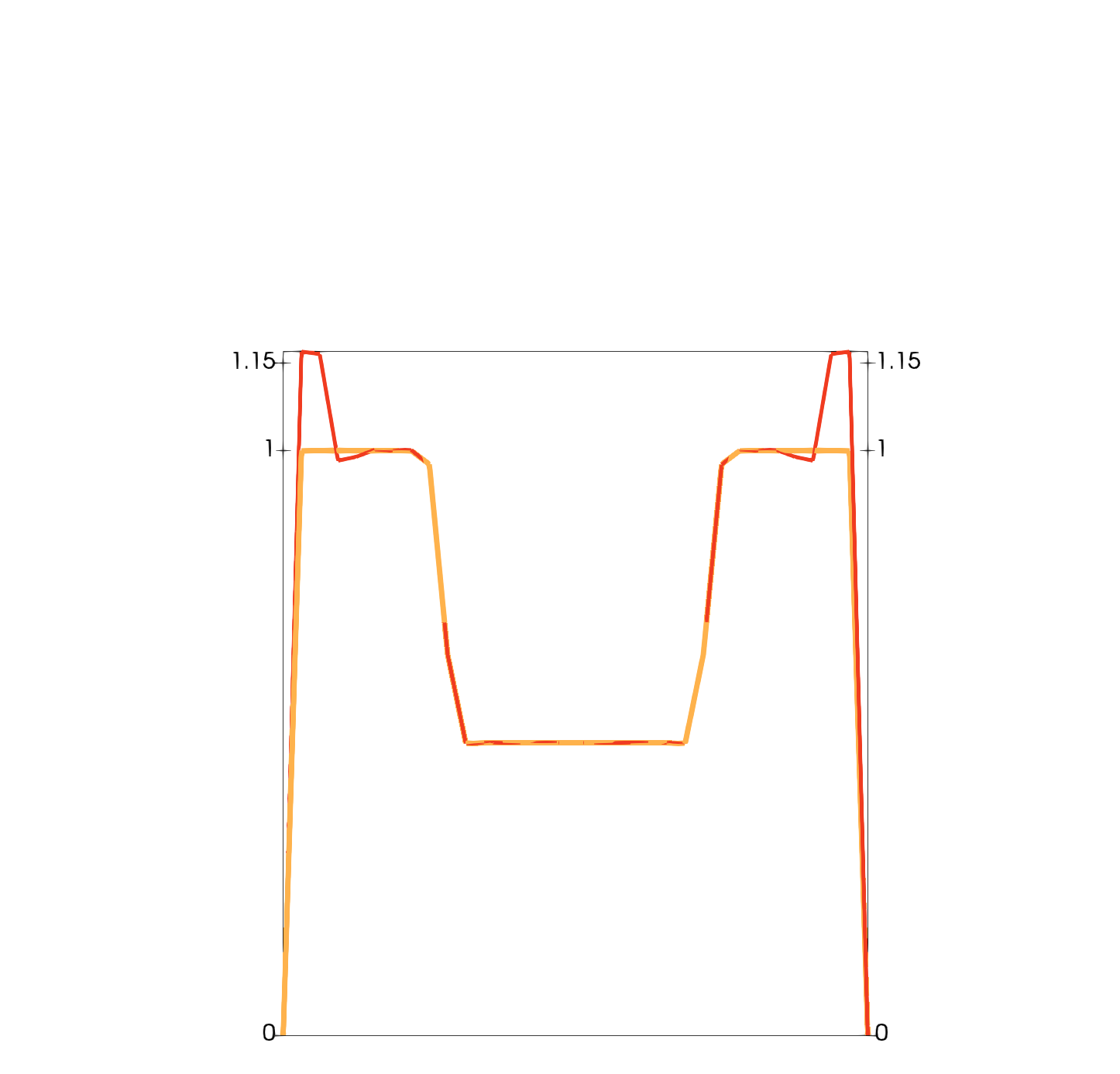}
        \caption{$\epsilon = 10^{-4}$, $k=1$.}
    \end{subfigure}%
    \begin{subfigure}[t]{0.3\textwidth}
        \includegraphics[width=\textwidth]{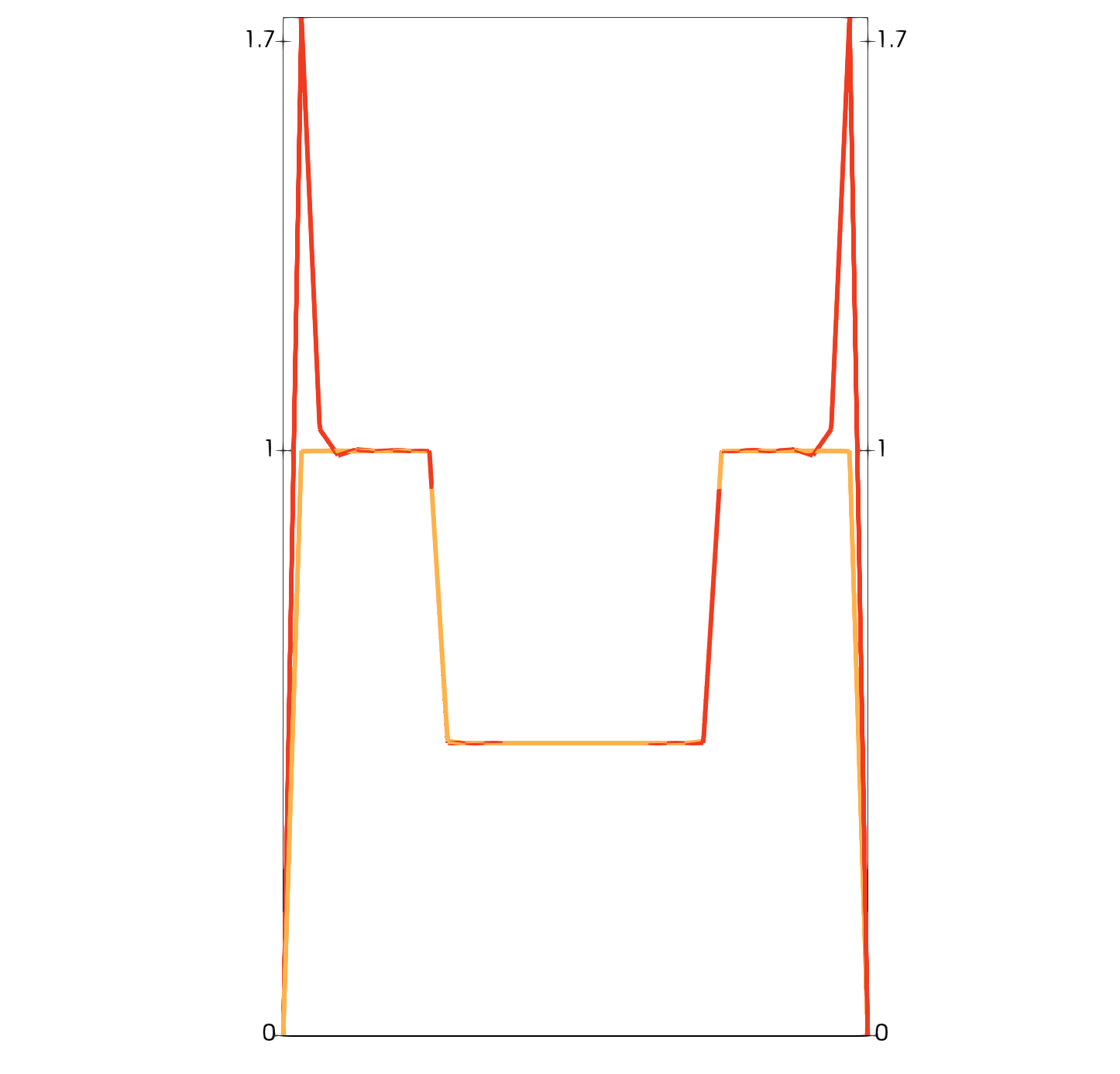}
        \caption{$\epsilon = 10^{-7}$, $k=1$.}
    \end{subfigure}
    \begin{subfigure}[t]{0.3\textwidth}
      \includegraphics[width=\textwidth]{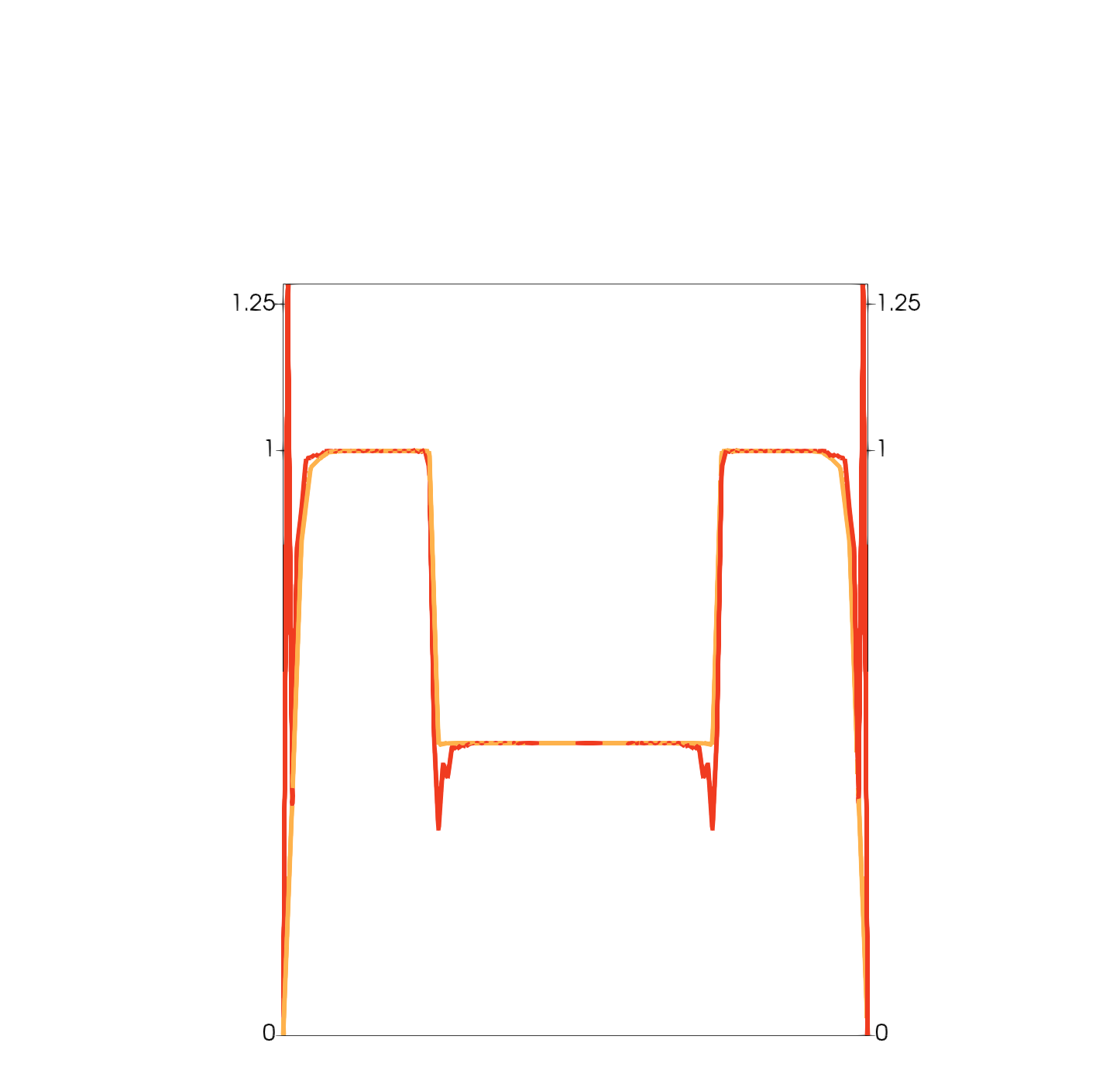}
        \caption{$\epsilon = 10^{-7}$, $k=2$.}
    \end{subfigure}
    \caption{\label{fig:intmin} A cross section taken about the $x=y$
      plane of the finite element solution, coloured red presenting
      oscillations at the boundary layer, and the approximation given
      by (\ref{eq:iterativeFEM}), coloured green with no oscillations
      at the boundary. Panels (A) and (B) are computed using $k=1$,
      while (C) depicts the approximations obtained using $k=2$.}
\end{figure*}

\subsection{Anisotropic diffusion with nonlinear reaction}

Consider the domain $\W = \W_1 \backslash \W_2$ with $\W_1 = (0,1)^2$,
$\W_2 = [\tfrac 49, \tfrac 59]^2$ and the problem
\begin{equation}
  \label{eq:nl}
  \begin{split}
    -\text{div}\qp{\epsilon \calD \nabla u} + u^3 &= f \text{ in } \W\,,
    \\
    u &=
    \begin{cases}
      0 \text{ on }\partial \W_1
      \\
      2 \text{ on }\partial \W_2
    \end{cases},
  \end{split}
\end{equation}
with $\epsilon = 10^{-5}$,
\begin{equation}
  \calD = \mathcal R \mathcal A \mathcal R^T,
  \quad
  \mathcal R
  = 
  \begin{bmatrix}
    \cos(\theta)& \sin(\theta)
    \\
    -\sin(\theta)& \cos(\theta)
  \end{bmatrix}
  \quad
  \mathcal A
  =
  \begin{bmatrix}
    100 & 0
    \\
    0& 1
  \end{bmatrix},
\end{equation}
and $\theta = -\tfrac \pi 6$. This is a challenging realisation of
(\ref{eq:pde}) with an anisotropic diffusion coefficient and nonlinear
reaction term, with $p=4$, posed over a nonconvex domain. In Figure
\ref{fig:nonlinear} we show the finite element solution, the {bound-preserving} solution and contour plots highlighting the oscillatory
nature of the finite element solution for this problem.

\begin{figure*}[h!]
    \centering
    \begin{subfigure}[t]{0.5\textwidth}
        \centering
        \includegraphics[width=1.\textwidth]{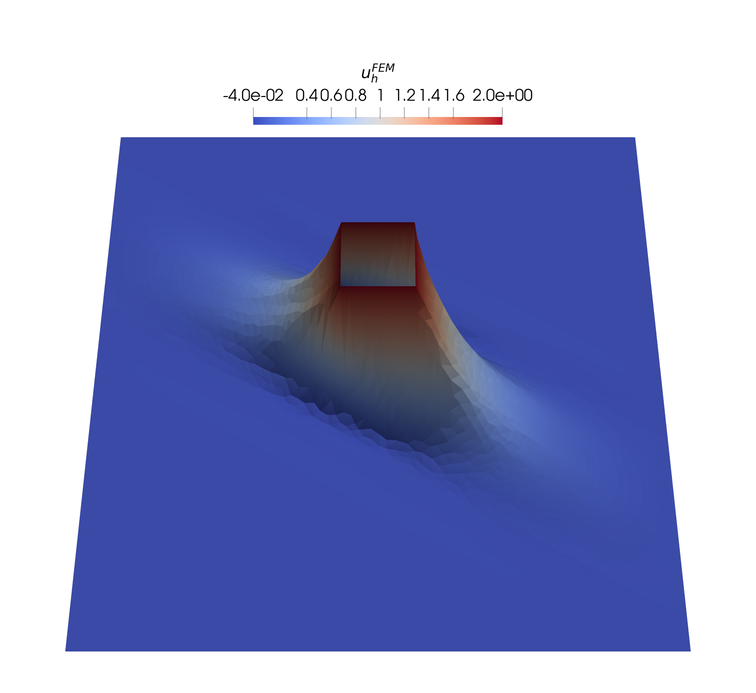}
        \caption{Finite element solution.}
    \end{subfigure}%
    ~
    \begin{subfigure}[t]{0.5\textwidth}
        \centering
        \includegraphics[width=1.\textwidth]{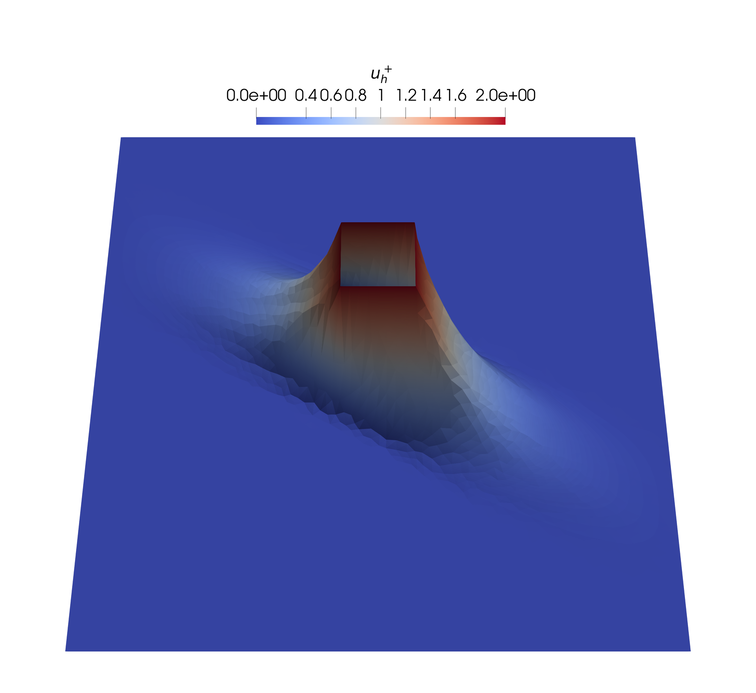}
        \caption{Bound preserving solution.}
    \end{subfigure}
    \\
    \begin{subfigure}[t]{0.5\textwidth}
        \centering
        \includegraphics[width=1.\textwidth]{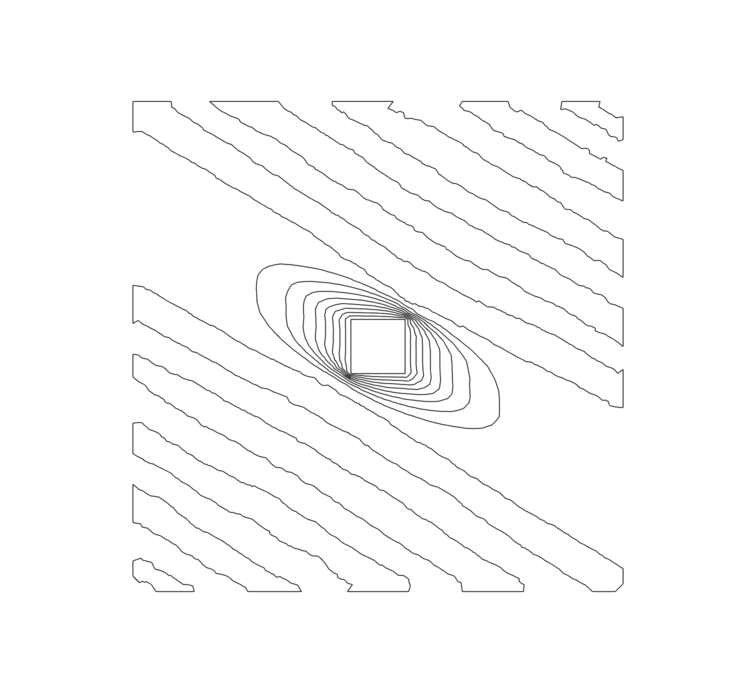}
        \caption{Contour plot of the finite element solution.}
    \end{subfigure}%
    ~
    \begin{subfigure}[t]{0.5\textwidth}
        \centering
        \includegraphics[width=1.\textwidth]{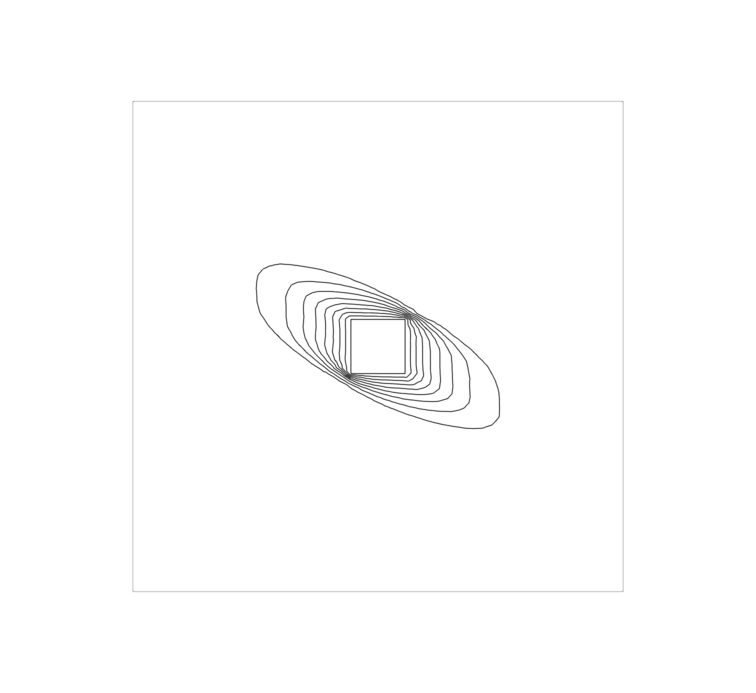}
        \caption{Contour plot of the bound preserving solution.}
    \end{subfigure}
    \caption{\label{fig:nonlinear} Approximations to the problem
      (\ref{eq:nl}) with $h \approx 0.02$ over a Delaunay mesh. Notice
      the finite element approximation exhibits oscillatory behaviour
      in contrast to the bound preserving method.}
\end{figure*}

\section{Concluding remarks}
We proposed an inexpensive and simple way to impose hard bounds on the range of a finite element solution. This is achieved through the definition of a nonlinear stabilised Galerkin approach, which is designed to provide the orthogonal projection into the closed convex set of physically admissible solutions satisfying hard bounds.  In an effort to highlight the key ideas, we have confined the presentation to linear and monotone semilinear reaction-diffusion equations. We stress, however, that the framework is general enough to allow other model problems, as well as bound-preserving variants of other known finite element methods, to be constructed following the methodology presented here. In particular, an interesting extension of the proposed methodology to convection-dominated problems is both relevant and, we believe, within reach.  Moreover, the extension to discontinuous Galerkin methods posed on general polygonal/polyhedral meshes is also conceivable. Both these, and other,  topics are ongoing,  and will be discussed elsewhere.

\section*{Acknowledgements} The work of GRB has been funded by the Leverhulme Trust through the Research Fellowship No. RF-2019-510. The work of AV is supported by the Italian GNCS and the MIUR PRIN 2017 NA-FROM-PDEs.

\bibliographystyle{abbrv} 
\bibliography{BGPV}
\end{document}